\documentclass[12pt]{amsart}
\usepackage{tkz-fct}  

\usepackage[margin=1.15in]{geometry}

\usepackage{amsmath,amscd,amssymb,amsfonts,latexsym,wasysym, mathrsfs, mathtools,hhline,color, bm}
\usepackage[all, cmtip]{xy}
\usepackage{comment}

\usepackage{url}

\definecolor{hot}{RGB}{65,105,225}

\usepackage[pagebackref=true,colorlinks=true, linkcolor=hot ,  citecolor=hot, urlcolor=hot]{hyperref}

\usepackage{ textcomp }
\usepackage{ tipa }

\usepackage{graphicx,enumerate}

\usepackage{enumitem}

\theoremstyle{plain}
\newtheorem{theorem}{Theorem}[section]
\newtheorem{proposition}[theorem]{Proposition}

\newtheorem{lm}[theorem]{Lemma}

\newtheorem{corollary}[theorem]{Corollary}

\newtheorem{lemma}[theorem]{Lemma}
\newtheorem{thrm}[theorem]{Theorem}

\theoremstyle{definition}

\newtheorem{definition}[theorem]{Definition}

\newtheorem{remark}[theorem]{Remark}

\newtheorem{ex}[theorem]{Example}
\newtheorem*{ex*}{Example}
\newtheorem{problem}{Problem}

\def\be{\begin{equation}}
\def\ee{\end{equation}}

\def\bt{\begin{thrm}}
\def\et{\end{thrm}}

\def\bc{\begin{cor}}
\def\ec{\end{cor}}

\def\br{\begin{rmk}}
\def\er{\end{rmk}}

\def\bp{\begin{prop}}
\def\ep{\end{prop}}

\def\bl{\begin{lm}}
\def\el{\end{lm}}

\def\bex{\begin{ex}}
\def\eex{\end{ex}}

\def\bd{\begin{defn}}
\def\ed{\end{defn}}

\newcommand\sH{{\mathcal H}}
\newcommand\sC{{\mathscr C}}

\newcommand\sF{{\mathcal F}}
\newcommand\sE{{\mathcal E}}

\newcommand\sM{{\mathcal M}}

\newcommand\sU{{\mathcal U}}
\newcommand\sV{{\mathcal V}}

\newcommand\sZ{\mathcal{Z}}

\newcommand\cc{{\mathbb{C}}}

\DeclareMathOperator{\Crit}{Crit}                  
                  
\DeclareMathOperator{\Jac}{Jac}
\DeclareMathOperator{\Lim}{Lim}
\DeclareMathOperator{\Proj}{Proj}

\DeclareMathOperator{\sFc}{\sF^{\centerdot}}

\DeclareMathOperator{\codim}{codim}              
\DeclareMathOperator{\reg}{reg}                  
\DeclareMathOperator{\sing}{sing}                  

\DeclareMathOperator{\Eu}{\mathrm{Eu}}

\DeclareMathOperator{\Zero}{Zero}

\DeclareMathOperator{\EDdeg}{EDdeg}
\DeclareMathOperator{\PEDdeg}{EDdeg_{proj}}

\def\bC{\mathbb{C}}
\def\RR{\mathbb{R}}
\def\cM{\mathcal{M}}

\def\bP{\mathbb{P}}

\def\cN{\mathcal{N}}

\def\lra{\longrightarrow}
\def\bQ{\mathbb{Q}}

\def\bZ{\mathbb{Z}}

\def\balpha{{\bm\alpha}}
\def\bbeta{{\bm\beta}}
\def\C{\mathbb{C}}

\title[]{A Morse theoretic approach to non-isolated singularities and applications to optimization}
\author{Laurentiu G. Maxim}
\address{Department of Mathematics,         University of Wisconsin-Madison,  480 Lincoln Drive, Madison WI 53706-1388, USA.}
\email {maxim@math.wisc.edu}\urladdr{https://www.math.wisc.edu/~maxim/}
\author{Jose Israel Rodriguez}
\address{Department of Mathematics,         University of Wisconsin-Madison,  480 Lincoln Drive, Madison WI 53706-1388, USA.}
\email {jose@math.wisc.edu}\urladdr{http://www.math.wisc.edu/~jose/}
\author{Botong Wang}
\address{Department of Mathematics,         University of Wisconsin-Madison,  480 Lincoln Drive, Madison WI 53706-1388, USA.}
\email {wang@math.wisc.edu}\urladdr{http://www.math.wisc.edu/~wang/}

\keywords{Euclidean distance degree, Euler characteristic, local Euler obstruction function, optimal solution,    stationary point,       maximum likelihood,    objective function}

\subjclass[2010]{13P25, 57R20, 90C26}

\begin{document}

\date{\today}

\begin{abstract} 
Let $X$ be a complex affine variety in $\mathbb{C}^N$, and let $f:\mathbb{C}^N\to \mathbb{C}$ be a polynomial function whose restriction to $X$ is nonconstant.  
For $g:\mathbb{C}^N \to \mathbb{C}$ a general linear function, 
we study the  limiting behavior of the critical points of the one-parameter family of
$f_t: =f-tg$ as $t\to 0$.
Our main result gives an expression of this limit in terms of critical sets of the restrictions of $g$ to the singular strata of $(X,f)$. 
We apply this result in the context of optimization problems.
For example, we consider nearest point problems (e.g., Euclidean distance degrees)
for affine varieties and a possibly nongeneric~data~point.

\end{abstract}

\maketitle

\section{Introduction}


The motivation for this work is to study nearest point problems for algebraic models and Euclidean distance degrees. 
For example, given a circle and a point $P$ outside its center,  there is a unique point on the circle which is closest to $P$, as seen in Figure~\ref{fig:dataZero}. 
%
However, if $P$ is taken to be the center, then every point on the circle is a closest point. 
Our aim is to understand such a special (non-generic) behavior on (arbitrary) algebraic varieties 
 by a limiting procedure on a set of critical points.
In terms of applied algebraic geometry, our results can be understood as describing what happens when genericity assumptions of statements on Euclidean distance degrees are removed (see Section~\ref{ss:ED}).
In optimization, our results state what happens as we take a regularization term to zero.

 \begin{figure}[htb!]
   \label{fig:dataZero}
 \centering
   \begin{picture}(153,183)
     {\includegraphics[width=0.37\columnwidth]{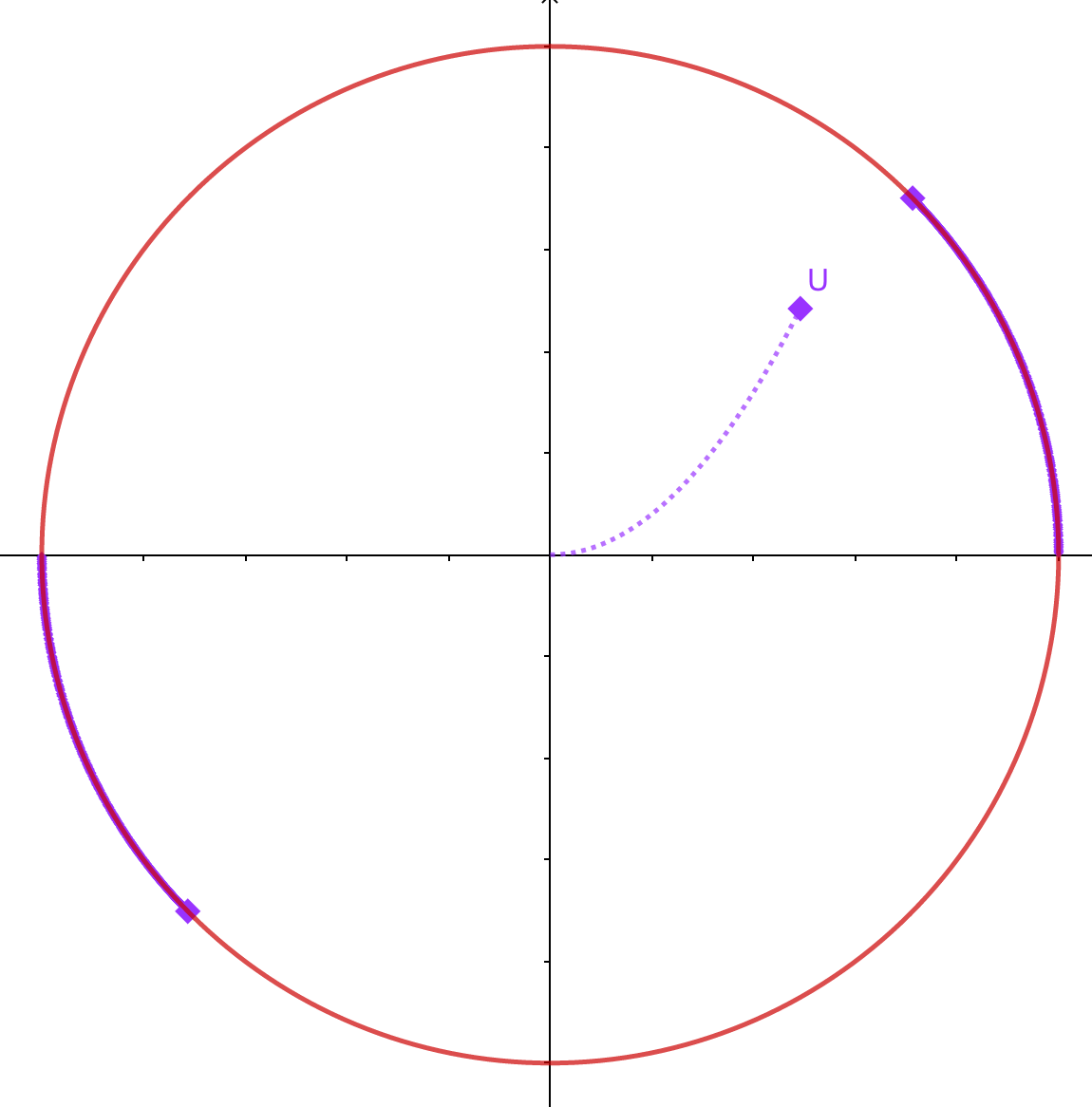}}
   \end{picture}
   \qquad\qquad
   \begin{picture}(97,183)(16,0)
     {\includegraphics[width=0.4\columnwidth]{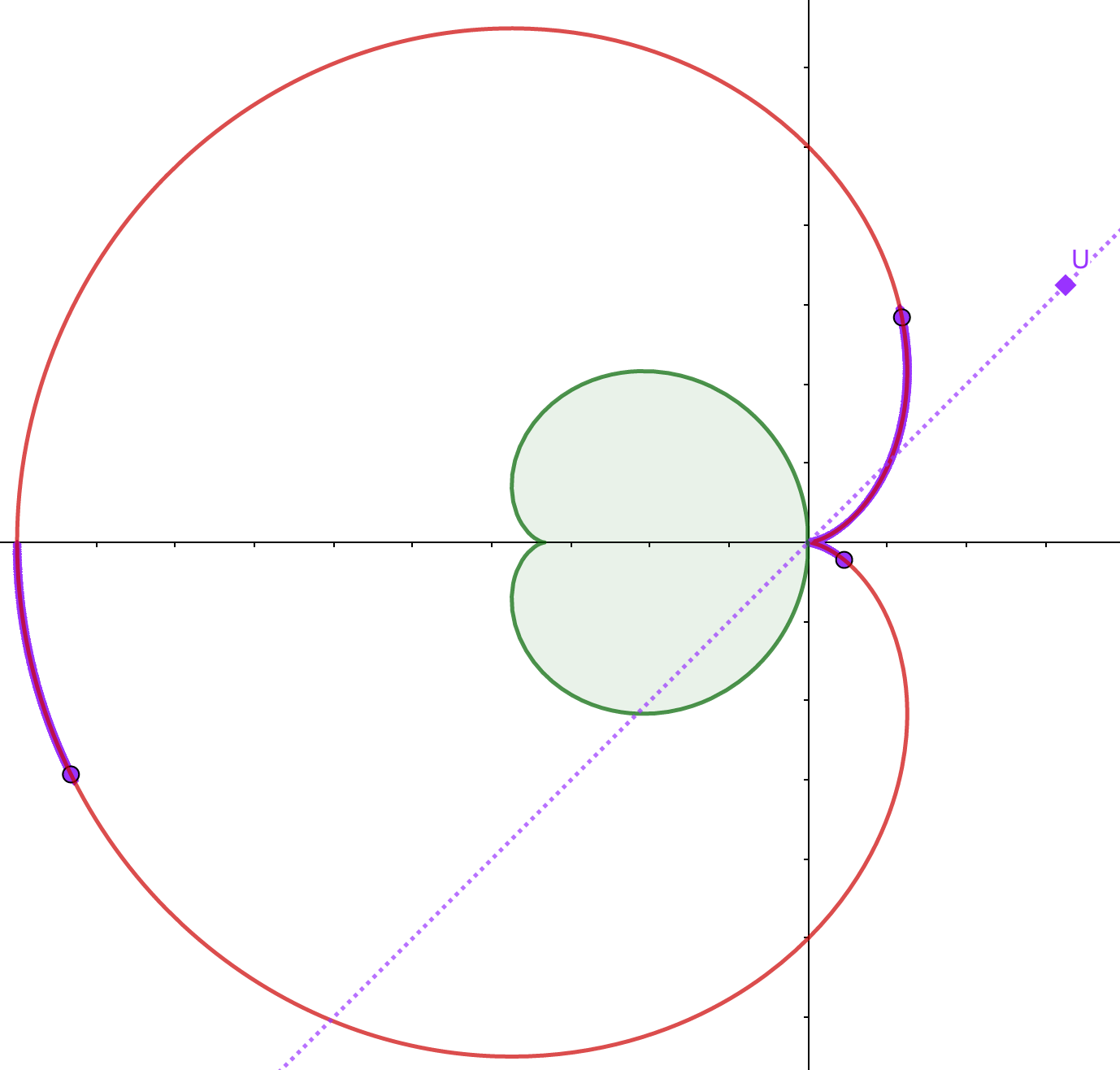}}
   \end{picture}
	\newline
   \caption{
   Each red curves is the set of real points of an algebraic variety $X$ and a 
   purple point on the curve is a 
   critical point of the distance function with respect to $U$. As $U$ moves along the dotted path, the critical points move along the purple arcs.    
   \texttt{Left}: 
    When $U$ is at the origin, every point on the circle $X$ is equidistant from $U$.  
   \texttt{Right}: 
   As $U$ approaches the origin, 
   the three critical points move along the cardioid curve $X$ and two of them come together. The green curve denotes the ED discriminant and when $U$ is in the shaded region there is only one real critical point on the smooth locus of $X$. 
   }
\end{figure}

Before stating the general result, we start with the following simple, but enlightening example. Let $X=\bC^N$, and let $f:\bC^N\to \bC$ be a polynomial function with isolated critical points $P_1, \ldots, P_l$. In this case, the singularity behavior of the function $f$ at each $P_i$ is governed by the Milnor number of $f$ at $P_i$ (see \cite{Mi}), which we denote by $n_i$. In particular, $f$ is a holomorphic Morse function (that is, it has only non-degenerate isolated critical points) if and only if each Milnor number $n_i$ is $1$. Fix a general linear function $g:\bC^N\to \bC$. Then $f_t\coloneqq f-tg$ is a holomorphic Morse function on $\bC^N$ for all but finitely many $t\in \bC$. The limit of the critical locus of $f_t$ has the following behavior as $t$ goes to 0. In a small neighborhood $U_i$ of $P_i$, there are $n_i$ non-degenerate critical points of $f_t$ for nonzero $t$ with small absolute value. As $t$ approaches zero, these critical points collide together to $P_i$. This process is the Morsification of $f$, which is a well-known result in singularity theory (see \cite[Appendix]{Br}). 

In general, we allow $X$ to be a possibly singular subvariety of $\bC^N$, and we allow $f$ to be any polynomial whose restriction to $X$ is nonconstant. If $g$ is a general linear function, then $$f_t\coloneqq f-tg$$ is a holomorphic Morse function on the smooth locus $X_{\reg}$ of $X$ for all but finitely many $t\in \bC$. We are interested in the limiting behavior of the set of critical points of $f_t\vert_{X_{\reg}}$ as $t$ approaches zero. 

In order to formulate our main result, let $X\subset \bC^N$ be a complex affine variety and let $f: \bC^N\to \bC$ be a polynomial function whose restriction to $X$ is nonconstant. Consider a stratification $X=\bigcup_{i\in I} X_i$ of $X$ into smooth locally closed subvarieties such that the Lagrangian cycles of the perverse vanishing cycle functors $^p\Phi_{f-c}([T_X^*\bC^N])$ are ``locally constant along $X_i$'' for all values of $c\in \bC$ and all $i \in I$. 
{Such a stratification of $X$ can be obtained explicitly as follows. As it will be explained in Section \ref{cc}, there exists a constructible complex $\sFc$ on $\bC^N$ with support on $X$, whose characteristic cycle is exactly the conormal space $T_X^*\bC^N$. We regard $\sFc$ as a constructible complex on $X$. The restriction $f|_X$ has only finitely many critical values in the stratified sense (see, e.g., \cite[Definition 4.2.7]{Di}), and for each such critical value $c \in \bC$ of $f|_X$ the (perverse) vanishing cycle functor $^p\Phi_{f-c}:D^b_c(X) \to D^b_c(X \cap \{f=c\})$ is constructible and supported in the stratified singular locus of $f$ (see, e.g., \cite[Proposition 4.2.8]{Di}). Choose a stratification of $X \cap \{f=c\}$ into smooth locally closed subvarieties with respect to which $^p\Phi_{f-c}(\sFc)$ is constructible. The required stratification of $X$ is then obtained by collecting all the strata in $X \cap \{f=c\}$ for each critical value $c$ of $f|_X$, together with a Whitney stratification of the complement of these critical fibers in $X$. Once such a stratification $X=\bigcup_{i\in I} X_i$ of $X$ is fixed, we 
have the following equality of Lagrangian cycles
\begin{equation}\label{eq_nii}
\sum_{c\in \bC}\,^p\Phi_{f-c}([T_X^*\bC^N])=\sum_{i\in I}n_{i}[T^*_{\overline{X_{i}}}\bC^N]
\end{equation}
for $n_i\in \bZ_{\geq 0}$.}  Notice that the sum on the left-hand side of \eqref{eq_nii} is a finite sum, since $f|_X$ has only finitely many critical values in the stratified sense, and $^p\Phi_{f-c}([T_X^*\bC^N])=0$ when $c$ is not a critical value. 
Moreover, it follows from work of Massey (see Theorem \ref{thm_positive}) that the coefficients $n_i$ are nonnegative.

By the characteristic cycle functor (see \eqref{eq_xy}), equation \eqref{eq_nii} amounts to express (up to signs), for each critical value $c$ of $f$, the constructible function $\varphi_{f-c}(\Eu_X)$ in terms of the basis of local Euler obstruction functions $\Eu_{\overline{X_{i}}}$ corresponding to closures of strata in $f=c$. Here, $\Eu_X$ denote the local Euler obstruction function introduced by MacPherson in \cite{M74}.
In general, an explicit calculation of the coefficients $n_i$ is difficult (see Example \ref{ex:rank}). 
However, when $f|_X$ has simple singularities, the vanishing cycle on the left-hand side of \eqref{eq_nii} can be  computed by hand, as the following examples show.
\begin{ex}\label{ex_milnor}
Suppose that $X$ is smooth and $f|_X$ has isolated critical points $P_1, \ldots, P_l$. Then we can take the stratification 
\[
X_0=X\setminus \{P_1, \ldots, P_l\}\quad \text{and} \quad X_i=\{P_i\}.
\]
The corresponding coefficients $n_i$ in \eqref{eq_nii} are computed directly as $n_0=0$, and $n_i$ is equal to the Milnor number of $f|_X$ at $P_i$, that is the length of the Artinian algebra $\mathcal{O}_{X, P_i}/\langle \frac{\partial f}{\partial z_1}, \ldots, \frac{\partial f}{\partial z_d}\rangle$, where $\mathcal{O}_{X, P_i}$ is the germ of holomorphic functions on $X$ at $P_i$ and $z_1, \ldots, z_d\in \mathcal{O}_{X, P_i}$ are the local  coordinates. 
\end{ex}

\begin{ex}\label{str_iso}
	Let $X\subset \bC^N$ be a possibly singular complex affine variety. Let $f: \bC^N\to \bC$ be a polynomial function whose restriction to $X$ is nonconstant and has only isolated critical points $P_1, \ldots, P_l$ in the stratified sense. 
	Then formula \eqref{eq_nii}, written in the language of constructible functions (see Section \ref{neva}), becomes:
	\be\label{unu}-\sum_{c \in \bC} \Phi_{f-c}((-1)^{\dim X}\Eu_X)=\sum_{i=1}^l n_i\Eu_{P_i}.\ee
	Fix $i \in \{1,\ldots,l\}$ and apply the equality of constructible functions in \eqref{unu} to the point $P_i$ to get:
	\be\label{niis} n_i=(-1)^{\dim X-1}\Phi_{f-f(P_i)}(\Eu_X)(P_i).\ee
Of course, if $X$ is smooth, then $\Eu_X=1_X$ and, via \eqref{Mi}, \eqref{eq10} and \eqref{16}, $n_i$ becomes the Milnor number of $f$ at $P_i$, as already mentioned in Example \ref{ex_milnor}.
\end{ex}



Let $g: \bC^N\to \bC$ be a general linear function, and write as above $f_t\coloneqq f-tg$. Our main result is the following:
\begin{theorem}\label{thm_main}
The limit of the critical points of $f_t$ satisfies
\begin{equation}\label{eq_main}
\lim_{t\to 0}\Crit(f_t|_{X_{\reg}})=\sum_{i\in I}n_{i}\cdot\Crit(g|_{X_{i}})
\end{equation}
where the symbol $\Crit$ denotes the set of critical points and the numbers $n_i$ are determined by formula \eqref{eq_nii}. 

\begin{remark}
The limit in \eqref{eq_main} is defined in Subsection \ref{ss_limit}. It is always well-defined in our setting, and it does not count the points going to infinity. 
\end{remark}

\end{theorem}

\begin{ex}\label{ex:specialization}
As in Example \ref{ex_milnor}, suppose that $X$ is smooth and $f|_X$ has only isolated critical points $P_1, \ldots, P_l$. Then Theorem \ref{thm_main} specializes to the well-known Morsification result that
\[
\lim_{t\to 0}\Crit(f_t)=\sum_{1\leq i\leq l} n_i P_i
\]
where $n_i$ is the Milnor number of $f|_X$ at $P_i$. 

Moreover, consider the situation of Example \ref{str_iso} of a possibly singular affine variety $X$, with $f$ having only isolated stratified singular points $P_1, \ldots, P_l$ on $X$. Theorem \ref{thm_main} specializes in this case to  \[
\lim_{t\to 0}\Crit(f_t)=\sum_{1\leq i\leq l} n_i P_i ,
\]
with $n_i$ computed as in formula \eqref{niis}.
\end{ex}

\begin{remark}
The left side of equation (\ref{eq_main}) does not count the points that go to infinity as $t\to 0$. To be precise, we say that \emph{no points of $\Crit(f_t|_{X_{\reg}})$ go to infinity} if 
\[\bigcup_{0<t\leq \epsilon}\Crit(f_t|_{X_{\reg}})\]
 is bounded in $\bC^N$ for sufficiently small $\epsilon \in \mathbb{R}_{>0}$. 
By (\ref{eq_main}), no points of $\Crit(f_t|_{X_{\reg}})$ go to infinity if and only if
\[|\Crit(f_t|_{X_{\reg}})|=\sum_{i\in I}n_i \cdot |\Crit(g|_{X_{i}})|,\]
for general $t$, where $|\cdot |$ denotes the cardinality of a set. 
\end{remark}

More generally, we can define the \emph{number of points of $\Crit(f_t|_{X_{\reg}})$ going to infinity} to be the number of points of $\Crit(f_t|_{X_{\reg}})$ outside of a sufficiently large ball centered at the origin for sufficiently small $t$. More precisely, it is the cardinality of $\Crit(f_t|_{X_{\reg}})\setminus B_r$ for $r\gg 0$, and $0<t\ll \frac{1}{r}$, where $B_r \subset \C^N$ is the ball of radius $r$ centered at the origin. We can give a topological interpretation of the number of points of $\Crit(f_t|_{X_{\reg}})$ going to infinity at $t$ goes to zero as follows. 

First, using a result of Seade, Tib\v{a}r and Verjovsky (see \cite[Equation (2)]{STV}), together with arguments similar to \cite[Section 3.3]{MRW2018}, we have:
\begin{theorem}
Let $X$ be any irreducible subvariety of $\bC^N$, and let $f$ be any polynomial function on $\bC^N$. For a general linear function $g$ on $\bC^N$, the number of critical points of $(f-g)|_{X_{\reg}}$ is equal to 
\[
(-1)^{\dim X}\chi({\Eu_X}|_{\sU})
\]
where $\sU$ is the complement of the hypersurface $\{f-g=c\}$ in $\bC^N$ for a general choice of $c\in \bC$. 
\end{theorem}
Together with Theorem \ref{thm_main}, this yields the following:
\begin{corollary}\label{cor:inf}
The number of points of $\Crit(f_t|_{X_{\reg}})$ going to infinity is equal to 
\begin{equation}\label{eq:inf}
(-1)^{\dim X}\chi({\Eu_X}|_{\sU})-\sum_{i\in I}n_{i}\cdot \big|\Crit(g|_{X_{i}})\big|
\end{equation}
where $|\cdot |$ denotes the cardinality of a set. 
\end{corollary}

As an immediate application of Corollary \ref{cor:inf} together with our  result from \cite[Theorem 1.3]{MRW2018}, 
we provide a new formula for the
Euclidean distance (ED) degree of an affine variety. 
In the previous literature, the Euclidean distance degree of an algebraic variety is described in terms of a distance function with respect to a generic data point. 
The following corollary (with a mild hypothesis regarding critical points at infinity), 
gives a formula for the ED degree in terms of 
 critical points of a general \emph{linear} function on strata $X_i$
 where $f$ is a
 distance function with respect to an \emph{arbitrary} data point.
 
 \begin{corollary}\label{cor:ED}
Fix a data point $(u_1,\dots,u_n)\in\mathbb{C}^n$ and an algebraic variety $X\subset\mathbb{C}^n$.
For $f=\sum_{i=1}^n(x_i-u_i)^2$, if no points of $\Crit(f-tg|_{X_{\reg}})$ go to infinity as $t\to 0$, 
then the Euclidean distance degree of $X$ equals 
$$\sum_{i\in I}n_{i}\cdot |\Crit(g|_{X_{i}})|.$$
\end{corollary}

To study the limiting behavior of the set of critical points, we use the work of Ginsburg \cite{G} on characteristic cycles. Another (possibly more direct) approach is to make use of Massey's results from \cite{Massey1}, which we learnt about as we were in the final stage of writing up this paper. For more details, see Remark \ref{remark_Massey}.

 \medskip

The paper is organized as follows. 
In Section~\ref{sec:2}, we introduce the notion of limit of sets, and recall the necessary background on constructible complexes and their characteristic cycles. In Section \ref{sec_Ginzburg}, we review Ginsburg's work of pushforward of characteristic cycles and the characteristic cycle of the nearby cycle functor. Our main result, Theoreom \ref{thm_main} is proved in Section \ref{sec:4}, while Section \ref{sec:5} is devoted to applications.

\medskip 

{\bf Acknowledgements} The authors thank J\"org Sch\"urmann for useful discussions and for bringing the references \cite{Massey1, Massey2} to their attention.  
L. Maxim is partially supported by the Simons Foundation Collaboration Grant \#567077. He also thanks the Sydney Mathematical Research Institute for hospitality and for providing him with excellent working conditions during the final stage of writing this paper.
J.~I. Rodriguez is partially supported by the College of Letters and Science, UW-Madison. 
B. Wang is partially supported by the NSF grant DMS-1701305 and by a Sloan Fellowship.


\section{Preliminaries}\label{sec:2}
In this section, we give a precise definition of the limit of sets. We also review the notion of characteristic cycles, nearby/vanishing cycles, and their relations.
\subsection{Limit of sets}\label{ss_limit}
We introduce here the notion of limit for a parametrized family of sets, which appears in the formulation of our main result, Theorem \ref{thm_main}.
\begin{definition}
Throughout this paper, by \textbf{a set of points}, we always mean a finite set with multiplicity. More precisely, fixing a ground set $S$, by a set of points $\cM$ of $S$, we mean a function $\cM: S\to \mathbb{Z}_{\geq 0}$ such that $\cM(x)=0$ for all but finitely many $x\in S$. We call $\cM(x)$ the \textbf{multiplicity} of $\cM$ at $x$. 
For two sets of points $\cM$ and $\cN$ of $S$, we write $\cM\geq \cN$, if $\cM(x)\geq \cN(x)$ for every point $x\in S$. 

Let $\phi: S\to T$ be a map of sets, and let $\cM$ be a set of points in $S$. Then $\phi(\cM)$ is a set of points in $T$ defined by
\[\phi(\cM)(y)=\sum_{x\in \phi^{-1}(y)}\cM(x).\]
\end{definition}
\begin{ex}
Any finite subset $T\subset S$ can be considered as a set of points $\cM^T$ in $S$, by setting 
\[
\cM^T(x)=
\begin{cases}
1, & x\in T,\\
0, & x\notin T.
\end{cases}
\]
\end{ex}

\begin{definition}
Fixing a Hausdorff space $S$ as the ground set, let $\cM_t$ be a family of sets of points of $S$, parametrized by $t\in \mathbb{C}^*$ (or more generally a punctured disc centered at the origin). We define \textbf{the limit of} $\cM_t$ as $t\to 0$, denoted by $\lim_{t\to 0}\cM_t$, to be the set of points given by:
\[
(  \lim_{t\to 0}\cM_t  )  (x)\coloneqq \varprojlim_{U}  \lim_{t\to 0}\sum_{y\in U}\cM_t(y),
\]
where $\varprojlim_U$ denotes taking the inverse limit over all open neighborhood of $x$. 
\end{definition}
\begin{remark}
The limit $\lim_{t\to 0}$ either exists as a finite set with multiplicity, or does not exist. If the limit $\lim_{t\to 0}$ exists, then for any $x\in S$, and for any sufficiently small neighborhood $U$ of $x$, then the limit $\lim_{t\to 0}\sum_{y\in U}\cM_t(y)$ exists as a finite number. 
\end{remark}

\begin{remark}
From now on, all the limits we work with are of algebraic nature. 
More precisely, $S$ is an algebraic variety, and there exists a (not necessarily irreducible) algebraic curve $C\in S\times \bC^*$, such that $\sM_t=p_S(C\cap S\times \{t\})$, where $p_S: S\times \bC^*\to S$ is the projection to the first factor. In this case, it is easy to see that the limit $\lim_{t\to 0}\sM_t$ always exists. 
\end{remark}

\begin{lemma}\label{lemma_top}
Let $\phi: S\to T$ be a proper continuous map between Hausdorff and locally compact spaces. Let $\sM_t$ be a family of sets with multiplicity parametrized by $t\in \bC^*$. Then
\begin{equation}\label{eq_phi}
\phi\big(\lim_{t\to 0}\sM_t\big)=\lim_{t\to 0}\phi(\sM_t)
\end{equation}
if both limits exist. 
\end{lemma}
\begin{proof}
The inequality 
\begin{equation}\label{eq_sets}
\phi\big(\lim_{t\to 0}\sM_t\big)\leq \lim_{t\to 0}\phi(\sM_t)
\end{equation}
is obvious, and does not require any compactness assumption. Now we prove the converse. 

Since the statement is local in $T$, we may assume that $\lim_{t\to 0}\phi(\sM_t)$ is supported at one point, that is $\lim_{t\to 0}\phi(\sM_t)=nQ$ for some $Q\in T$ and $n\in \bZ_{\geq 0}$. To show the converse of \eqref{eq_sets}, it suffices to show both sides have the same multiplicity at $Q$. 

Since $T$ is locally compact, there exists an arbitrarily small compact neighborhood $V$ of $Q$ and $\epsilon_V>0$, such that $\sM_t\cap \phi^{-1}(V)$ consists of $n$ points for any $0<t<\epsilon_V$. 
Let 
\[\lim_{t\to 0}\sM_t=\sum_{i\in J} m_iP_i.\]
By \eqref{eq_sets}, we have $P_i\in \phi^{-1}(Q)\subset \phi^{-1}(V)$. By definition, for any $x\in V$, there exists a neighborhood $U_x$ of $x$ in $V$ and $\epsilon_x>0$, such that $\sM_t\cap U_{P_i}$ consists of $m_i$ points, and $\sM_t\cap U_x$ is empty if $x\notin \{P_i| i\in J\}$. Since $S$ is Hausdorff, we can also assume that $U_{P_i}$ are pairwise disjoint for $i\in J$. Since $\phi$ is proper, $\phi^{-1}(V)$ is compact. Hence we can cover $\phi^{-1}(V)$ by finitely many $U_x$. Let $\epsilon$ be the smallest $\epsilon_x$ among all $x$ appearing in the index of the above covering. Then for any $0<t<\epsilon$, the set with multiplicity $\sM_t\cap \phi^{-1}(V)$ consists of $\sum_{i\in J}m_i$ points. Thus, $\sum_{i\in J}m_i=n$, that is $\phi\big(\lim_{t\to 0}\sM_t\big)$ and $\lim_{t\to 0}\phi(\sM_t)$ have the same multiplicity at $Q$. 
\end{proof}

\subsection{Constructible complexes and characteristic cycles}\label{cc}
A sheaf $\sF$ of $\bC$-vector spaces on a variety $M$ is {\it constructible} if there exists a finite stratification $M=\sqcup_j S_j$ of $X$ into locally closed smooth subvarieties (called {\it strata}), such that the restriction of $\sF$ to each stratum $S_j$ is a $\bC$-local system of finite rank. A complex $\sF^{\centerdot}$ of sheaves of $\bC$-vector spaces on $M$ is called constructible if its cohomology sheaves $\sH^i(\sF^{\centerdot})$ are all constructible. 
Denote by $D^b_c(M)$ the bounded derived category of constructible complexes (with respect to some stratification) on $M$, i.e., one identifies constructible complexes containing the same cohomological information.

By associating characteristic cycles to constructible complexes on a smooth variety $M$ (e.g., see \cite[Definition 4.3.19]{Di} or \cite[Chapter IX]{KS}), one gets a functor 
$$CC:K_0(D^b_c(M)) \lra LCZ(T^*M)$$
on the Grothendieck group of $\bC$-constructible complexes, where $LCZ(T^*M)$ is the free abelian group spanned by the irreducible conic Lagrangian cycles in the cotangent bundle $T^*M$. Recall that any element of $LCZ(T^*M)$ is of the form $\sum_k n_k [T^*_{Z_k}M]$, for some $n_k \in \bZ$ and $Z_k$ closed irreducible subvarieties of $M$. Here, if $Z$ is a closed irreducible subvariety of $M$ with smooth locus $Z_{\reg}$, its conormal bundle $T^*_{Z}M$ is defined as the closure in $T^*M$ of $T^*_{Z_{\reg}}M$.
One can then define a group isomorphism
$$T:LCZ(T^*M) \lra Z(M)$$ to the group $Z(M)$ of algebraic cycles on $M$ by:
$$\sum_k n_k [T^*_{Z_k}M] \longmapsto \sum_k (-1)^{\dim Z_k} n_k Z_k.$$

Let $CF(M)$ be the group of algebraically constructible functions on a complex algebraic variety $M$, i.e., the free abelian group generated by indicator functions $1_Z$ of closed irreducible subvarieties $Z \subset M$. To any constructible complex $\sF^{\centerdot} \in D^b_c(M)$ one associates a constructible function $\chi_{st}(\sF^{\centerdot})\in CF(M)$ by taking stalkwise Euler characteristics, i.e.,
$$\chi_{st}(\sF^{\centerdot})(x):=\chi(\sF^{\centerdot}_x)$$
for any $x \in X$. For example, $\chi_{st}(\bC_M)=1_M$. Another important example of a constructible function is the {\it local Euler obstruction} function $\Eu_M$ of MacPherson \cite{M74}, which is an essential ingredient in the definition of Chern classes for singular varieties.
Since the Euler characteristic is additive with respect to distinguished triangles, one gets an induced group homomorphism (in fact, an epimorphism)
$$\chi_{st}:K_0(D^b_c(M)) \lra CF(M).$$
Moreover, since the class map $D^b_c(M) \to K_0(D^b_c(M))$ is onto, $\chi_{st}$ is already an epimorphism on $D^b_c(M)$.

When $Z$ is a closed subvariety of $M$, we may regard the function $\Eu_Z$ as being defined on  all of $M$ by setting $\Eu_Z(x)=0$ for $x \in M \setminus Z$. In particular, we may consider the group homomorphism
\begin{equation}\label{eq_ZCF}
\Eu:Z(M) \lra CF(M)
\end{equation}
defined on an irreducible cycle $Z$ by the assignment $Z \mapsto \Eu_Z$, and then extended by $\bZ$-linearity. A well-known result (e.g., see \cite[Theorem 4.1.38]{Di} and the references therein) states that the homomorphism
$\Eu:Z(M) \to CF(M)$
is an isomorphism.

The Euler obstruction function enters into the formulation of the {\it local index theorem}, which in the above notations and for $M$ smooth asserts the existence of the following commutative diagram  (e.g., see \cite[Section 5.0.3]{S} and the references therein):
\begin{equation}\label{eq_xy}
\xymatrix{
K_0(D^b_c(M)) \ar[d]_{CC} \ar[r]^{\chi_{st}} & CF(M) \ar[d]^{\Eu^{-1}}_{\cong}  \\
LCZ(T^*M) \ar[r]^T_{\cong} & Z(M)
}
\end{equation}
In particular, one can associate a characteristic cycle to any constructible function $\varphi \in CF(M)$ by the formula
$$CC(\varphi):=T^{-1} \circ \Eu^{-1}(\varphi).$$
For example, if $Z$ is a closed irreducible subvariety of $M$, one has:
$$CC(\Eu_Z)=(-1)^{\dim Z}[T^*_Z M].$$
Note also that $$CC(\sF^{\centerdot})=CC(\chi_{st}(\sF^{\centerdot}))$$ for any constructible complex $\sF^{\centerdot} \in D^b_c(M)$.

\subsection{Nearby and vanishing cycle functors}\label{neva}
Let $M$ be a complex manifold, and let $f:M \to \Delta$ be a holomorphic map 
to a disc, with $i:f^{-1}(0) \hookrightarrow M$ the inclusion of the zero-fiber. The
canonical fiber $M_{\infty}$ of $f$ is defined by
$$M_{\infty}:=M \times_{\Delta^*} \hbar,$$ where $\hbar$ is the
complex upper-half plane (i.e., the universal cover of the
punctured disc via the map $z \mapsto \exp(2\pi i z)$). Let $k:M_{\infty} \hookrightarrow M$ be the induced map.
The \emph{nearby cycle functor of $f$}, $\Psi_f: D^b_c(M) \to D^b_c(f^{-1}(0))$ is defined by
\begin{equation} \Psi_f(\sFc):=i^*Rk_*k^* \sFc.\end{equation}  The
\emph{vanishing cycle functor} $\Phi_f: D^b_c(M) \to D^b_c(f^{-1}(0))$ is the
cone on the comparison morphism $i^*\sFc \to
\Psi_f(\sFc)$, that is, there exists a canonical morphism
$can:\Psi_f(\sFc) \to \Phi_f(\sFc)$ such that
\begin{equation}\label{sp}
i^*\sFc \to \Psi_f(\sFc) \overset{can}{\to} \Phi_f(\sFc)
\overset{[1]}{\to}\end{equation} is a distinguished triangle in
$D^b_c(f^{-1}(0))$.

It follows directly from the definition that for $x \in X_0$, \begin{equation}\label{Mi}
H^j(M_{f,x};\bQ)=\mathcal{H}^j(\Psi_f \bQ_X)_x \ \ \ {\rm and} \ \ \ 
\widetilde{H}^j(M_{f,x};\bQ)=\mathcal{H}^j(\Phi_f \bQ_X)_x,\end{equation}
where $M_{f,x}$ denotes the Milnor fiber  of $f$ at $x$. 

It is also known that the
shifted functors $^p \Psi_f:=\Psi_f[-1]$ and $^p \Phi_f:=\Phi_f[-1]$
take perverse sheaves on $M$ into perverse sheaves on the 
zero-fiber $f^{-1}(0)$ (e.g., see \cite[Theorem 6.0.2]{S}).

By repeating the above constructions for the function $f-c$, one gets functors 
$$\Psi_{f-c}, \Phi_{f-c}: D^b_c(M) \to D^b_c(f^{-1}(c)),$$
provided that $\{f=c\}$ is a nonempty hypersurface.

The nearby cycle functor descends to a functor on the category of constructible functions, see, e.g., \cite{Ve} or \cite[Section 4]{Sch12}. In other words, the constructible function $\chi_{st}\big(\Psi_f(\sF)\big)$ only depends on the function $\chi_{st}(\sF)$. Therefore, $\Psi_f$ induces a linear map, which we also denote by $\Psi_f$, 
\begin{equation}
\Psi_f:CF(M) \to CF(f^{-1}(0)), 
\end{equation}
where we regard elements of $CF(f^{-1}(0))$ as constructible functions om $M$ with support on $f^{-1}(0)$.
In fact, the above linear map $\Psi_f$ can be defined directly as follows: 
\begin{equation}\label{eq10}\Psi_f(\alpha)(x)=\chi(\alpha \cdot 1_{M_{f,x}}).\end{equation}
In particular, \begin{equation}\label{eq11}\Psi_f(1_M)=\mu \in CF(f^{-1}(0)),\end{equation}
where $\mu:f^{-1}(0) \to \bZ$ is the  constructible function defined by the rule:
\begin{equation}\label{eq12}\mu(x):=\chi(M_{f,x}),\end{equation} for all $x \in f^{-1}(0)$. Note that 
$$
\mu=\chi_{st}(\Psi_f\bQ_X).
$$

By analogy with (\ref{sp}), one defines a vanishing cycle functor on constructible functions,
\begin{equation}
\Phi_f:CF(M) \to CF(f^{-1}(0))\subset CF(M), 
\end{equation}
by setting \be\label{16}\Phi_f(\alpha):=\Psi_f(\alpha)-\alpha\vert_{f^{-1}(0)}.\ee

\begin{remark}
By \eqref{eq_xy}, the characteristic cycle functor $CC:CF(M)\overset{\cong}\to LCZ(T^*M)$ allows one to regard the nearby and vanishing cycle functors $\Psi_f, \Phi_f$ as functors on conic Lagrangian cycles in the cotangent bundle $T^*M$. This will be the way we view nearby and vanishing cycle functors for the rest of this paper. 
\end{remark}

\subsection{Pushforward, Pullback, Attaching triangle}
Let $M$ be a complex manifold as above, and let $f:M \to \Delta$ be a holomorphic map 
to a disc. Let $i:f^{-1}(0) \hookrightarrow M$ and $j:U=M \setminus f^{-1}(0) \hookrightarrow M$ be the inclusion maps of the zero-fiber and of its complement, respectively. Recall that for any $\sFc \in D^b_c(M)$, there is an attaching triangle in $D^b_c(M)$:
\begin{equation}\label{at}
j_!j^*\sFc \to \sFc \to i_*i^*\sFc 	\overset{[1]}{\to}
\end{equation}
with $i_*=i_!$.

\begin{lemma}\label{l29}
Under the above notations, we have:
\begin{equation}
\chi_{st}(j_!j^*\sFc)=\chi_{st}(Rj_*j^*\sFc) \in CF(M).	
\end{equation}
\end{lemma}
\begin{proof}
Since $j^*j_! \simeq j^*Rj_*\simeq id$, we see that the restrictions of the complexes $j_!j^*\sFc$ and $Rj_*j^*\sFc$ to $U$ are quasi-isomorphic, so they have the same stalks over $U$. At points in $f^{-1}(0)$, the complex $j_!j^*\sFc$ has zero stalks. So it remains to show that $\chi((Rj_*j^*\sFc)_x)=0$ for all $x \in f^{-1}(0)$. 
Next note that for any  $x\in f^{-1}(0)$ and $k \in \mathbb{Z}$, we have: $$H^k((Rj_*j^*\sFc)_x)\cong \mathbb{H}^k(B_x; Rj_*j^*\sFc) \cong \mathbb{H}^k(B_x \setminus f^{-1}(0); j^*\sFc),$$ for $B_x$ a small enough ball in $M$ centered at $x$. Therefore,
$$\chi((Rj_*j^*\sFc)_x)=\chi(B_x \setminus f^{-1}(0); j^*\sFc).$$
Finally, using \cite[Corollary 4.1.23, Remark 4.1.24]{Di} and the implied additivity of Euler characteristics from (\ref{at}), one gets that:
$$\chi((Rj_*j^*\sFc)_x)= \chi(B_x;\sFc) - \chi(B_x \cap f^{-1}(0); i^*\sFc)=\chi(\sFc_x)- \chi(\sFc_x)=0,$$
thus completing the proof.
	\end{proof}
\begin{remark}
The above lemma is also a special case of \cite[Example 6.0.17(1)]{S}, and it can be deduced from the distinguished triangle
$$j_!j^*\sFc \to Rj_*j^*\sFc \to i_*i^*Rj_*j^*\sFc \overset{[1]}{\to}$$
(which is obtained by applying \eqref{at} to $Rj_*j^*\sFc$ instead of $\sFc$, and using $j^*Rj_*\simeq id$), by noting that (cf. \cite[(6.37)]{S}): $$[i_*i^*Rj_*j^*\sFc]=0\in K_0(D^b_c(M)).$$
\end{remark}

When coupled with the local index theorem, Lemma \ref{l29} yields the following.
\begin{corollary} In the above notations, we have:
\begin{equation}\label{eq_cc}
	CC(j_!j^*\sFc)=CC(Rj_*j^*\sFc).
\end{equation}
\end{corollary}

It is well known (e.g., see \cite[Section 2.3]{S}) that all the usual functors in sheaf theory, which respect the corresponding category of constructible complexes of sheaves, induce by the epimorphism $\chi_{st}$ well-defined group homomorphisms on the level of constructible functions. This was already indicated above for the nearby and vanishing cycle functors, and the same applies for the functors $i_*$, $i^*$, $j_!$, $j^*$, $Rj_*$, which on the level of constructible functions are denoted by  $i_*=i_!$, $i^*$, $j_!$, $j^*$, $j_*$. In particular, by \eqref{eq_xy}, these functors can also be considered as functors on conic Lagrangian cycles in the cotangent bundle $T^*M$ (with support in a certain subvariety, if needed).
\begin{proposition}\label{prop_cycles}
In the above notations, let $\Lambda$ be a conic Lagrangian cycle in $T^*M$. Then
\[
^p\Psi_f(\Lambda)=j_*j^*(\Lambda)+\, ^p\Phi_f(\Lambda)-\Lambda \in LCZ(T^*M).
\]
\end{proposition}
\begin{proof}
Since the characteristic cycle functor $CC: K_0\big(D^b_c(M)\big)\to LCZ\big(T^*M\big)$ is surjective, the distinguished triangle \eqref{sp} implies that
\[
\Psi_f(\Lambda)=i^*(\Lambda)+\Phi_f(\Lambda),
\]
as an identity of Lagrangian cycles in $LCZ(T^*M)$, with support in $f^{-1}(0)$. In particular, we identify $i^*(\Lambda)$ and $i_*i^*(\Lambda)$ in $LCZ(T^*M)$. 
Furthermore, the distinguished triangle \eqref{at} yields that
\[
i_*i^*(\Lambda)+j_! j^*(\Lambda)=\Lambda 
\]
and, by \eqref{eq_cc}, we have
\[
j_! j^*(\Lambda)=j_* j^*(\Lambda). 
\]
Combining the above three equations, we get:
\[
\Psi_f(\Lambda)=\Lambda-j_* j^*(\Lambda)+\Phi_f(\Lambda).
\]
Notice that as functors of Lagrangian cycles (just as on $K_0(D^b_c(M))$), $^p\Psi_f(\Lambda)$ and $^p\Phi_f(\Lambda)$ are equal to the negative of $\Psi_f(\Lambda)$ and $\Phi_f(\Lambda)$, respectively. Thus, the assertion follows from the above equation. 
\end{proof}


\section{The characteristic cycle of nearby and vanishing cycles functors}\label{sec_Ginzburg}
In this section, we review Ginsburg's work \cite{G} on the pushforward of characteristic cycles and the characteristic cycle of nearby cycle functor. 

Let $M$ be a complex manifold. Given any holomorphic function $f: M\to \bC$, let $U$ be the complement of the hypersurface $f^{-1}(0)$ in $M$. Given any conic Lagrangian cycle $\Lambda$ in $T^*U$, Ginsburg defined the pushforward of $\Lambda$ by the open inclusion map $j: U\to M$, denoted by 
$\underset{s\to 0}{\Lim}\, \Lambda_s^\#$, as follows. 
For any $s\in \bC^*$, define the non-conic Lagrangian cycle $\Lambda_s^\#$ by
\begin{equation}\label{def_limit}
\Lambda_s^\#=\Lambda+s\cdot d\log f=\big\{\xi+s(d\log f)(x)\mid (x, \xi)\in \Lambda\big\}.
\end{equation}
The total space of the family $\Lambda_s^\#$ forms a closed subvariety $\Lambda^\#$ of $T^*U\times \bC^*$. We denote its closure in $T^*M\times \bC$ by $\overline{\Lambda^\#}$. To define the characteristic cycle $\Lim_{s\to 0}\Lambda_s^\#$, one first takes the scheme-theoretic intersection $\overline{\Lambda^\#}\cap (T^*M\times \{0\})$, and then considers the cycle obtained by taking the irreducible components of this intersection with the multiplicities given by the scheme structure. One obtains in this way a conic Lagrangian cycle $\Lim_{s\to 0}\Lambda_s^\#$ in $T^*M$. In view of the following result of Ginsburg, one should regard it as the pushforward of $\Lambda$ by the open embedding $j: U\to M$. 
\begin{theorem}\cite[Theorem 3.2]{G}\label{thm_G1}
Let $\Lambda$ be a conic Lagrangian cycle in $T^*U$. Then
\begin{equation}
Rj_*(\Lambda)=\underset{s\to 0}{\Lim}\,\Lambda_s^\#.
\end{equation}
\end{theorem}

Ginsburg also computed the characteristic cycle of the nearby cycle of a constructible complex by using a similar construction. Denote the projection $T^*M\times \bC\to T^*M$ by $w$. 

\begin{proposition}\cite[Proposition 2.14.1]{G}\label{prop_G}
	Under the above notations, over a neighborhood of $f^{-1}(0)\subset M$,
	\begin{enumerate}
		\item the set $\overline{\Lambda^\#}$ is an analytic variety of dimension $\dim M+1$;
		\item if $(\xi_x, s)\in \overline{\Lambda^\#}$ and $f(x)=0$, then $s=0$;
		\item the restriction $w|_{\overline{\Lambda^\#}}: \overline{\Lambda^\#}\to T^*M$ is a closed embedding. 
	\end{enumerate}
\end{proposition}

\begin{corollary}\label{cor_G}\cite[Corollary 2.14.2]{G}
	For any conic Lagrangian subvariety $\Lambda$ of $T^*U$, the closure of
	\[
	w(\Lambda^\#)=\{\xi_x+s(d\log f)(x)\mid \xi_x\in \Lambda, s\in \bC^*\}
	\]
	in $T^*M$ is equal to $w(\overline{\Lambda^\#})$.
\end{corollary}

Denote the pushforward cycle $w_*(\Lambda^\#)$ by $\Lambda^\natural$. Denote by $\Lim_{f\to 0}\Lambda^\natural$ the specialization of $\Lambda^\natural$ to $f^{-1}(0)$. By Corollary \ref{cor_G}, $\Lambda^\natural$ is equal to the variety $w(\Lambda^\#)$ and $\Lim_{f\to 0}\Lambda^\natural$ is the schematic restriction of the variety $w(\overline{\Lambda^\#})$ to $T^*M|_{f^{-1}(0)}$. 

\begin{theorem}\cite[Theorem 5.5]{G}\label{thm_G2}
	Let $\Lambda$ be a conic Lagrangian cycle in $T^*M$. Then,
	\[
	^p\Psi_f(\Lambda)=\underset{f\to 0}{\Lim}\,({\Lambda|_{T^*U}})^\natural,
	\]
	where $^p\Psi_f$ is the perverse nearby cycle functor. 
\end{theorem}

It follows from Theorem \ref{thm_G1} and Theorem \ref{thm_G2} that if $\Lambda$ is an irreducible conic Lagrangian subvariety in $T^*M$, then both $j_*j^*(\Lambda)$ and $^p\Psi_f(\Lambda)$ are effective. The same is true for the vanishing cycle functor by the following result. 

\begin{theorem}\cite[Theorem 2.10]{Massey1}\label{thm_positive}
Let $\Lambda$ be an irreducible conic Lagrangian subvariety in $T^*M$. Then $^p\Phi_f(\Lambda)$ is an effective conic Lagrangian cycle. 
\end{theorem}

See Remark \ref{remark_Massey} for a brief discussion around \cite[Theorem 2.10]{Massey1}. 

\begin{remark}
Theorem 2.10 in \cite{Massey1} is formulated in the language of sheaves. Nevertheless, since the characteristic cycle of a bounded constructible complex only depends on the associated constructible function, in view of \eqref{eq_xy} the argument also works for Lagrangian cycles. See also \cite[Remark 1.3]{Massey2} for a correction of sign errors. 
\end{remark}

\section{The limit of critical points}\label{sec:4}
Let $X\subset \bC^N$ be an irreducible subvariety with smooth locus $X_{\reg}$. Let $f: \bC^N\to \bC$ be a polynomial map. 
Let $g: \bC^N\to \bC$ be a general linear function. We will study the critical locus of $f_t|_{X_{\reg}}$ as $t$ goes to zero, where $f_t=f-tg$. 

We introduce some notations that we will use throughout this section. We fix a stratification $X=\bigsqcup_{i\in I}X_i$ of $X$ into smooth locally closed subvarieties.  Let $$p: T^*\bC^N\to \bC^N$$ be the natural projection. Give any algebraic 1-form $\omega$ on $\bC^N$, let $\Gamma_{\omega}$ be the image of the 1-form $\omega$ in $T^*\bC^N$.

\begin{lemma}\label{lemma_eq}
Let
\[
\Gamma_{X_{\reg}, f}=T^*_{X_{\reg}}\bC^N-\Gamma_{df}=\big\{(x, \eta)\in T^*\bC^N\bigm| x\in X_{\reg},\hspace{2pt} \eta+df|_x\in T^*_{X_{\reg}}\bC^N\big\}.
\]
Then for any $t\in \bC$, we have
\begin{equation}\label{eq_intersection}
p(\Gamma_{X_{\reg}, f}\cap \Gamma_{tdg})=\Crit(f_t|_{X_{\reg}}).
\end{equation}
\end{lemma}
\begin{proof}
A point $x\in X_{\reg}$ is a critical point of $f_t$ if and only if the cotangent vector $df_t$ at $x$ is contained in $T^*_{X_{\reg}}\bC^N$. This is equivalent to $\Gamma_{df-tdg}$ intersects $T^*_{X_{\reg}}\bC^N$ in the fiber over $x$. By definition,
\[
p\big(\Gamma_{df-tdg}\cap T^*_{X_{\reg}}\bC^N\big)=p(\Gamma_{X_{\reg}, f}\cap \Gamma_{tdg}).
\]
Therefore, the assertion in the lemma follows. 
\end{proof}


We fix a general linear function $g: \bC^N\to \bC$. Let $\sC$ be the intersection of $\Gamma_{dg}\times \bP^1$ and the closure of $\overline{\Lambda^\#}$ in ${T^*\bC^N}\times \bP^1$, where $\Lambda=T^*_X\bC^N|_{T^* U}=T^*_{X\cap U} U$ and $\overline{\Lambda^\#}$
is the closed subvariety of $T^*_X\bC^N\times \bC$ defined in Section \ref{sec_Ginzburg}. 
\begin{remark}
	The curve $\sC$ is a lifting of the polar curve (see \cite[Definition 7.1.1]{Tibar}) on $X$ to $T^*\bC^N\times \bP^1$. 
\end{remark}
\begin{lemma}\label{lemma_irred}
	The curve $\sC$ is equal to the closure of ${\Lambda^\#}\cap \big(\Gamma_{dg}\times \bC\big)$ in ${T^*\bC^N}\times \bP^1$. 
\end{lemma}
\begin{proof}
Notice that $\Lambda^\#$ is Zariski open and dense in its closure in ${T^*\bC^N}\times \bP^1$. Thus, for general a choice of $g$, the intersection ${\Lambda^\#}\cap \big(\Gamma_{dg}\times \bC\big)$ is also Zariski open and dense in $\sC$. Thus, the lemma follows. 
\end{proof}

Let $\pi: {T^*\bC^N}\times \bP^1\to \bC^N$ be the composition of the projection to the first factor and the natural cotangent bundle map $T^*\bC^N\to \bC^N$. Denote by ${\pi}_{\sC}$ the restriction of ${\pi}: {T^*\bC^N}\times \bP^1\to \bC^N$ to $\sC$. By definition, $\sC$ is closed in $\Gamma_{dg}\times \bP^1$. Clearly, the restriction map $\pi|_{\Gamma_{dg}\times \bP^1}: \Gamma_{dg}\times \bP^1\to \bC^N$ is a $\bP^1$ bundle map. Therefore, the map ${\pi}_{\sC}: \sC\to \bC^N$ is the composition of a closed embedding $\sC\to \Gamma_{dg}\times \bP^1$ and a proper projection $\Gamma_{dg}\times \bP^1\to \bC^N$. Since a closed embedding is proper and the composition of proper maps is also proper, we have the following lemma. 
\begin{lemma}\label{lemma_cproper}
The map ${\pi}_{\sC}: \sC\to \bC^N$ is proper. 
\end{lemma}

As a subvariety of $T^*\bC^N\times \bP^1$, we can consider $f$ and $s$ as functions on $\sC$ by abuse of notations: 
\begin{enumerate}
\item The regular function $f$ is the composition $\sC\hookrightarrow {T^*\bC^N}\times \bP^1\xrightarrow{{\pi}}\bC^N\xrightarrow{f}\bC$.
\item The rational function $s$ is the composition of $\sC\hookrightarrow {T^*\bC^N}\times \bP^1\to \bP^1$, where the second arrow is the projection to the second factor. Recall that $s$ is the coordinate of the line $\bC$, and hence it extends to a rational function on $\bP^1$. 
\end{enumerate}


\begin{proposition}\label{prop_pole}
In a neighborhood of $\{f=0\}$ in $\sC$, the rational function $s$ is a regular function. In other words, the zero locus of $f$ on $\sC$ does not intersect the pole locus of $s$ on $\sC$. 
\end{proposition}
\begin{proof}
Recall that $\overline{\Lambda^\#}$ is a closed subvariety of $T^*\bC^N\times \bC$, and by Lemma \ref{lemma_irred}, $\sC$ is the closure of $\overline{\Lambda^\#}\cap \big(\Gamma_{dg}\times \bC\big)$ in ${T^*\bC^N}\times \bP^1$. Thus, it suffices to show that the intersection $\overline{\Lambda^\#}\cap \big(\Gamma_{dg}\times \bC\big)$ is closed in a neighborhood of $\{f=0\}$ in ${T^*\bC^N}\times \bP^1$. This is a consequence of Proposition~\ref{prop_G}~(3), i.e., the restriction of $w: T^*\bC^N\times \bC\to T^*\bC^N$ to $\overline{\Lambda^\#}$ is a closed embedding in a neighborhood of $\{f=0\}$. 

In fact, since $\Gamma_{dg}$ is closed in $T^*\bC^N$, $\Gamma_{dg}\times \bC$ is closed in $T^*\bC^N\times \bC$, and hence $\overline{\Lambda^\#}\cap\big(\Gamma_{dg}\times \bC\big)$ is closed in $\overline{\Lambda^\#}$. Thus, 
\begin{equation*}\label{eq_proper}
\overline{\Lambda^\#}\cap\big(\Gamma_{dg}\times \bC\big) \hookrightarrow \overline{\Lambda^\#} \xrightarrow{w}  T^*\bC^N
\end{equation*}
is the composition of a closed embedding and a proper map (Proposition~\ref{prop_G}~(3)) in a neighborhood of $\{f=0\}$. Therefore, the above composition is a proper map. By the definition of $\sC$, the above composition factors through the natural inclusion map 
\begin{equation}\label{eq_inclusion}
\overline{\Lambda^\#}\cap\big(\Gamma_{dg}\times \bC\big)\hookrightarrow \sC.
\end{equation}
If a composition of maps of algebraic varieties is proper, then the first map must be proper (see e.g. \cite[Corollary 4.8 (e)]{Hartshorne}). Therefore, the open inclusion \eqref{eq_inclusion} is proper, and hence an isomorphism in a neighborhood of $\{f=0\}$. 
\end{proof}

By definition, the pole {locus} of $s$ as a rational function on $\sC$ is equal to the complement of $\overline{\Lambda^\#}\cap\big(\Gamma_{dg}\times \bC\big)$. Therefore, the above proposition is also equivalent to $\overline{\Lambda^\#}\cap\big(\Gamma_{dg}\times \bC\big)=\sC$ in a neighborhood of $\{f=0\}$. 
\begin{corollary}\label{cor_inj}
The map $\pi_\sC: \sC\to \bC^N$ is injective in a neighborhood of $\{f=0\}$. 
\end{corollary}
\begin{proof}
By the arguments preceding the corollary, it suffices to show that the restriction of $\pi: {T^*\bC^N}\times \bP^1\to \bC^N$  to $\overline{\Lambda^\#}\cap\big(\Gamma_{dg}\times \bC\big)$ is injective. The restriction factors through the projection $w: T^*\bC^N\times \bC\to T^*\bC^N$. So the map $\overline{\Lambda^\#}\cap\big(\Gamma_{dg}\times \bC\big)\to \bC^N$ factors as
\begin{equation}\label{eq_composition}
\overline{\Lambda^\#}\cap\big(\Gamma_{dg}\times \bC\big)\to w(\overline{\Lambda^\#})\cap \Gamma_{dg}  \to \bC^N
\end{equation}
By Proposition~\ref{prop_G}~(3), the restriction of $w$ to $\overline{\Lambda^\#}$ is injective. Hence the first map in \eqref{eq_composition} is injective. The second map in \eqref{eq_composition} is injective, because the map $\Gamma_{dg}\to \bC^N$ is an isomorphism. Therefore, the restriction of $\pi$ to $\overline{\Lambda^\#}\cap\big(\Gamma_{dg}\times \bC\big)$, which is equal to the composition of \eqref{eq_composition}, is injective. 
\end{proof}

Let $\widetilde{\sC}\to \sC$ be the normalization of $\sC$. Then for any nonzero rational function $h$ on $\sC$, its pullback to $\widetilde{\sC}$ defines an effective Weil divisors $\Zero_{\tilde{\sC}}(h)$ on $\widetilde{\sC}$ as its zero divisor. We denote the pushforward of $\Zero_{\tilde{\sC}}(h)$ to $\sC$ by $\Zero_{\sC}(h)$.

\begin{proposition}\label{prop_plimit}
Under the above notations, as sets with multiplicity
\begin{equation}\label{eq_pi}
{\pi}_{\sC}\big(\Zero_{\sC}(f/s)\big)=\lim_{t\to 0}\Crit(f_t|_{X_{\reg}})
\end{equation}
in a neighborhood of $\{f=0\}$. 
\end{proposition}

Before proving the proposition, we make the following observation. By abuse of notations, we can consider $f$ and $s$ as regular functions on the affine space $T^*\bC^N\times \bC$. Fixing a nonzero complex number $t$, we have a hypersurface $\{f=ts\}$ in $T^*\bC^N\times \bC$. Recall that $w: T^*\bC^N\times \bC\to T^*\bC^N$ and $p: T^*\bC^N\to \bC^N$ are the natural projections, and $\pi:T^*\bC^N\times \bC\to \bC^N$ is their composition. 

\begin{lemma}\label{lemma_pw}
 Under the above notations, for any fixed $t\in \bC^*$, we have
\[
p(\Gamma_{X_{\reg}, f}\cap \Gamma_{tdg})=\pi\Big({\Lambda^\#}\cap \big(\Gamma_{dg}\times \bC^*\big)\cap \{f=ts\}\Big).
\]
\end{lemma}
\begin{proof}
It is straightforward to check one by one that the following conditions are equivalent for a point $x\in X_{\reg}$.
\begin{enumerate}
\item $x\in \pi\Big({\Lambda^\#}\cap \big(\Gamma_{dg}\times \bC^*\big)\cap \{f=ts\}\Big)$.
\item The restriction of the 1-form $s\cdot \frac{df}{f}-dg$ to $X$ vanishes at $x$ for $s=\frac{f(x)}{t}$.
\item The restriction of the 1-form $df-tdg$ to $X$ vanishes at $x$. 
\item  $x\in p\big(\Gamma_{X_{\reg}, f}\cap \Gamma_{tdg}\big)$. 
\end{enumerate}
Thus, the assertion in the lemma follows. 
\end{proof}

\begin{proof}[Proof of Proposition \ref{prop_plimit}] 
By Lemma \ref{lemma_eq} and Lemma \ref{lemma_pw}, we have 
\begin{equation*}
\begin{split}
\lim_{t\to 0}\Crit(f_t|_{X_{\reg}})&=\lim_{t\to 0}\;p(\Gamma_{X_{\reg}, f}\cap \Gamma_{tdg})\\
&=\lim_{t\to 0}\;\pi\Big({\Lambda^\#}\cap \big(\Gamma_{dg}\times \bC^*\big)\cap \{f=ts\}\Big)\\
&=\lim_{t\to 0}\;\pi\Big({\Lambda^\#}\cap \big(\Gamma_{dg}\times \bC^*\big)\cap \{f/s=t\}\Big)
\end{split}
\end{equation*}
in a neighborhood of $\{f=0\}$ in $\bC^N$. 

As before, consider $f/s$ as a rational function on $\sC$. By Lemma \ref{lemma_irred}, the intersection $\Lambda^\#\cap \big(\Gamma_{dg}\times \bC^*\big)$ is a Zariski open and dense subset of $\sC$. Therefore, for all but finitely many $t\in \bC$, we have
\[
{\Lambda^\#}\cap \big(\Gamma_{dg}\times \bC^*\big)\cap \{f/s=t\}=\sC\cap \{f/s=t\}.
\]
Combining the above two equations, we have
\[
\lim_{t\to 0}\Crit(f_t|_{X_{\reg}})=\lim_{t\to 0}\;\pi_{\sC}\big(\sC\cap \{f/s=t\}\big)
\]
in a neighborhood of $\{f=0\}$ in $\bC^N$. By Lemma \ref{lemma_cproper} and Lemma \ref{lemma_top}, we have
\[
\lim_{t\to 0}\;\pi_{\sC}\big(\sC\cap \{f/s=t\}\big)=\pi_{\sC}\Big(\lim_{t\to 0}\big(\sC\cap \{f/s=t\}\big)\Big).
\]
Clearly, 
\[
\lim_{t\to 0}\big(\sC\cap \{f/s=t\}\big)=\Zero_\sC(f/s).
\]
Combining the above three equations, we have
\[
\lim_{t\to 0}\Crit(f_t|_{X_{\reg}})=\pi_{\sC}\big(\Zero_\sC(f/s)\big)
\]
in a neighborhood of $\{f=0\}$ in $\bC^N$. 
\end{proof}

Let $X\subset \bC^N$ and $f: \bC^N\to \bC$ as before, and let $^p \Psi_{f}$ be the perverse nearby cycle functor from conic Lagrangian cycles on $\bC^N$ to the ones supported on $\{f=0\}$.
Let $U$ be the complement of $\{f=0\}$ in $\bC^N$, and let $j: U\to \bC^N$ be the inclusion map. 
We fix a stratification
\[
X\cap \{f=0\}=\sqcup_{i\in I_0}S_i
\]
of $X\cap \{f=0\}$ into locally closed smooth subvarieties such that
\begin{equation}\label{eq_as1}
^p\Psi_{f}\big([T_X^*\bC^N]\big)=\sum_{i\in I_0}m'_{i}[T^*_{\overline{S_{i}}}\bC^N],
\end{equation}
and
\begin{equation}\label{eq_as2}
Rj_{*}\big([T_X^*\bC^N]|_U\big)=[T_X^*\bC^N]+\sum_{i\in I_0}l'_{i}[T^*_{\overline{S_{i}}}\bC^N],
\end{equation}
with $m'_i, l'_i\in \bZ_{\geq 0}$ for all $i\in I_0$. 

\begin{proposition}\label{prop_crit}
Under the above notations, counting multiplicities yields:
\begin{equation}\label{eq_pimi}
\pi_{\sC} \big(\Zero_{\sC}(f)\big)=\sum_{i\in I_0}m'_i \cdot \Crit(g|_{S_i})
\end{equation}
and
\begin{equation}\label{eq_pili}
\pi_{\sC} \big(\Zero_{\sC}(s)\big)=\sum_{i\in I_0}l'_i \cdot \Crit(g|_{S_i})
\end{equation}
in a neighborhood of $\{f=0\}$ of $\bC^N$. 
\end{proposition}
\begin{proof}
We will derive the statements in the proposition from Theorem \ref{thm_G1} and Theorem \ref{thm_G2} using similar arguments as in the proof of Proposition \ref{prop_plimit}. 

Considering $f$ as a regular function on the curve $\sC$, we have
\[
\Zero_{\sC}(f)=\lim_{c \to0}\{x\in \sC\mid f(x)=c\},
\]
where the limit is taken in $\sC$. Equivalently, considering $f$ as a regular function on $\overline{T^*\bC^N}\times \bP^1$ and $\{f=c\}$ as a hypersurface of $\overline{T^*\bC^N}\times \bP^1$, we have
\[
\Zero_{\sC}(f)=\lim_{c\to 0}\sC\cap \{f=c\},
\]
where the limit is taken in $\sC$. 
By Lemma \ref{lemma_irred}, $\overline{\Lambda^\#}\cap \big(\Gamma_{dg}\times \bC\big)$ is a nonempty Zariski open subset of $\sC$. Thus, 
\[
\lim_{c\to 0}\sC\cap \{f=c\}=\lim_{c\to 0}\overline{\Lambda^\#}\cap \big(\Gamma_{dg}\times \bC\big)\cap \{f=c\},
\]
where both limits are taken in $\sC$. 
Combining the above two equations, we have
\begin{equation}\label{eq_2pi}
\pi_{\sC} \big(\Zero_{\sC}(f)\big)=\pi\big(\lim_{c\to 0}\overline{\Lambda^\#}\cap \big(\Gamma_{dg}\times \bC\big)\cap \{f=c\}\big).
\end{equation}
Since the restriction of ${\pi}: {T^*\bC^N}\times \bP^1\to \bC^N$ to $\Gamma_{dg}\times \bP^1$ is proper,  Lemma \ref{lemma_top} implies that
\begin{equation}\label{eq_limpi}
\pi\big(\lim_{c\to 0}\overline{\Lambda^\#}\cap \big(\Gamma_{dg}\times \bC\big)\cap \{f=c\}\big)=\lim_{c\to 0}\pi\big(\overline{\Lambda^\#}\cap \big(\Gamma_{dg}\times \bC\big)\cap \{f=c\}\big)
\end{equation}
where the first limit is taken in ${T^*\bC^N}\times \bP^1$ and the second limit is taken in $\Gamma_{dg}\times \bP^1$. 

Recall that in Section \ref{sec_Ginzburg}, $w: T^*\bC^N\times \bC\to T^*\bC^N$ is the natural projection, and $\Lambda^\natural$ is equal to the pushforward $w_*(\overline{\Lambda^\#})$. Therefore, 
\begin{equation}\label{eq_pip}
\pi\big(\overline{\Lambda^\#}\cap \big(\Gamma_{dg}\times \bC\big)\cap \{f=c\}\big)=p\big(\Lambda^\natural\cap \Gamma_{dg}\cap \{f=c\}\big)
\end{equation}
where $p:T^*\bC^N\to \bC^N$ is the cotangent bundle map. Since the restriction of $p: T^*\bC^N\to \bC^N$ to $\Gamma_{dg}$ is an isomorphism, in particular proper, by Lemma \ref{lemma_top}, we have
\begin{equation}\label{eq_pcomm}
p\big(\lim_{c\to 0}\Lambda^\natural\cap \Gamma_{dg}\cap \{f=c\}\big)=\lim_{c\to 0}p\big(\Lambda^\natural\cap \Gamma_{dg}\cap \{f=c\}\big)
\end{equation}
where the first limit is in $T^*\bC^N$ and the second limit is in $\bC^N$. 

Recall that $\lim_{f\to 0}\Lambda^\natural$ is the schematic restriction of the variety $w(\overline{\Lambda^\#})$ to $T^*M|_{f^{-1}(0)}$. Since $\Lambda=T^*_X\bC^N$, by Theorem \ref{thm_G2}, we have 
\[
\underset{f\to 0}{\Lim}\,\Lambda^\natural= \,^p\Psi_f([T^*_X\bC^N])
\]
which by assumption \eqref{eq_as1} is equal to $\sum_{i\in I_0}m'_i \cdot \Crit(g|_{S_i})$. Since $g:\bC^N\to \bC$ is a general linear function, $\Gamma_{dg}$ intersects $T^*_{\overline{S_i}}\bC^N$ transversally and it also intersects $\Lambda^\natural\cap \{f=c\}$ transversally for all but finitely many $c\in \bC$. Therefore,
\begin{equation}\label{eq_limmi}
\lim_{c\to 0}\Lambda^\natural\cap \Gamma_{dg}\cap \{f=c\}=\sum_{i\in I_0}m'_i \cdot T^*_{\overline{S_i}}\bC^N\cap \Gamma_{dg}
\end{equation}
as sets with multiplicity, where the limit is taken in $T^*\bC^N$. 

Finally, equality \eqref{eq_pimi} follows from equations \eqref{eq_2pi}, \eqref{eq_limpi}, \eqref{eq_pip}, \eqref{eq_pcomm} and \eqref{eq_limmi}. The proof of equality \eqref{eq_pili} is similar.
The only difference is that, in this case, the term $[T^*_X\bC^N]$ in \eqref{eq_as2} does not contribute to the right side of \eqref{eq_pili}. In fact, since $f$ is nonconstant on $X$, the intersection $T^*_X\bC^N\cap \{f=0\}$ is of dimension at most $N-1$, and hence for a general $g$, the intersection $T^*_X\bC^N\cap \Gamma_g$ is empty in a sufficiently small neighborhood of $\{f=0\}$. 
\end{proof}

\begin{corollary}\label{cor_zero}
In a neighborhood of $\{f=0\}$ of $\sC$, as sets with multiplicity (or Weil divisors), we have
\[
\Zero_{\sC}(f/s)=\Zero_{\sC}(f)-\Zero_{\sC}(s). 
\]
\end{corollary}
\begin{proof}
It suffices to show that in a neighborhood of $\{f=0\}$, the underlying set $\Zero_{\sC}(f)$ does not contain any pole of $s$ and $\Zero_{\sC}(f)\geq \Zero_{\sC}(s)$. 

The first part follows from Proposition~\ref{prop_pole}. Now, we prove the second part. 
By Proposition \ref{prop_cycles} and Theorem \ref{thm_positive}, we have 
\[
m_i'-l_i'=n_i'\geq 0
\]
for very $i\in I_0$. Thus, by Proposition \ref{prop_crit}, 
\[
\pi_{\sC} \big(\Zero_{\sC}(f)\big)\geq \pi_{\sC} \big(\Zero_{\sC}(s)\big)
\]
as sets of multiplicity. Since $\pi_\sC$ is injective in a neighborhood of $\{f=0\}$ (Corollary~\ref{cor_inj}), we have $\Zero_{\sC}(f)\geq \Zero_{\sC}(s)$. 
\end{proof}

Before proving Theorem \ref{thm_main}, we prove a local version of the theorem. 
\begin{theorem}\label{thm_local}
Let $X\cap \{f=0\}=\bigsqcup_{i\in I_0}S_i$ be a stratification of $X\cap \{f=0\}$ as discussed in the paragraphs before Example \ref{ex_milnor}.  In particular, equations \eqref{eq_as1} and \eqref{eq_as2} hold. Then 
\begin{equation}\label{eq_local}
\lim_{t\to 0}\Crit(f_t|_{X_{\reg}})=\sum_{i\in I_0}n'_{i}\cdot\Crit(g|_{S_{i}})
\end{equation}
in an analytic neighborhood of $\{f=0\}$ in $\bC^N$. Moreover, the coefficients $n'_i$ are determined by the following formula,
\begin{equation}\label{eq_ni}
^p\Phi_{f}([T_X^*\bC^N])=\sum_{i\in I_0}n'_{i}[T^*_{\overline{S_{i}}}\bC^N].
\end{equation}
\end{theorem}
\begin{proof}
By Proposition \ref{prop_cycles}, we have 
\[
m_i'-l_i'=n_i'
\]
for very $i\in I_0$. Now the assertion in the theorem follows from Corollary \ref{cor_zero}, Proposition \ref{prop_plimit} and Proposition \ref{prop_crit}.\end{proof}

\begin{proof}[Proof of Theorem \ref{thm_main}]
To prove Theorem \ref{thm_main}, it suffices to show the assertions hold in a neighborhood of $\{f=c\}$ for every $c\in \bC$. This follows from Theorem \ref{thm_local} with $f$ replaced by $f-c$. 
\end{proof}

\begin{remark}\label{remark_Massey}
As we shall now explain, it is also possible to derive our results from \cite{Massey1} instead of using \cite{G}. A topological interpretation of the vanishing cycle of a conic Lagrangian cycle is obtained in \cite[Theorem 2.10]{Massey1}. Let $\Lambda$ be an irreducible conic Lagrangian subvariety of $T^*\bC^N$ and let $f: \bC^N\to \bC$ be a polynomial function. Blow up $T^*\bC^N$ along $\Gamma_{df}$, the image of the 1-form $df$. Let $\tilde{\Lambda}$ be the strict transformation of $\Lambda$, and let $E$ be the exceptional divisor. The natural isomorphism between $\Gamma_{df}$ and $\bC^N$ induces an isomorphism between $E$ and the projective bundle $\Proj(T^*\bC^N)$. Under this isomorphism,
\[
E\cap \tilde{\Lambda}=\sum_{c\in \bC}\Proj\big(\, ^p\Phi_{f-c}(\Lambda)\big)
\]
where the first intersection is considered as a schematic intersection counting multiplicities. 

We are interested in the case when $f|_X$ has positive dimensional critical locus, which corresponds to a positive dimensional intersection of $\Gamma_{df}$ and $T^*_X\bC^N$. The above approach of Massey is exactly the deformation to normal cone (see, e.g., \cite[Chapter 5]{Fulton}), which is designed to construct intersection cycles when the set-theoretic intersection has more than expected dimensions. See also \cite[Part IV]{Massey3} for some discussion related to L\^{e}-Vogel cycles. 
\end{remark}

\section{Applications and examples}\label{sec:5}
%
%
 
\subsection{$X$ is an affine space}
The first class of examples we consider are when $X=\mathbb{C}^n$, $f$ is a polynomial function, and $g$ is a general linear function.

\begin{ex}
The following illustrates a special case of Examples~\ref{ex:specialization}.
Consider a general linear function $g: \cc \to \cc$ and the function
\[
f : \cc \to \cc,\quad f(x)=x^4  - 4x^3,
\]
The function $f$ has a  critical point at zero and at three, which we denote by  $X_1$ and $X_2$ respectively. 
For general $t$, the function $f_t:=f-tg$ has three distinct critical points.
We have 
$\lim_{t\to 0}\Crit(f_t|_{X_{\reg}}) =\{X_1,	X_2\} $ 
and see that as $t\to 0$ two of the three critical points come together at $X_1$  and the other  has multiplicity one as shown in Figure~\ref{fig:uni}.

 \begin{figure}[htb!]
   \label{fig:uni}
 \centering
     {\includegraphics[width=0.3\columnwidth]{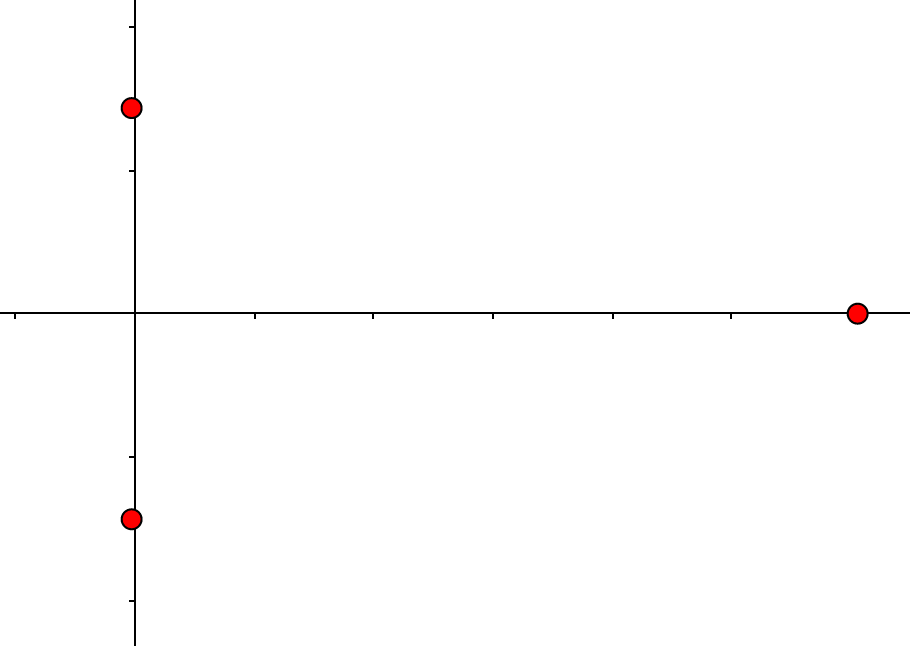}}
     {\includegraphics[width=0.3\columnwidth]{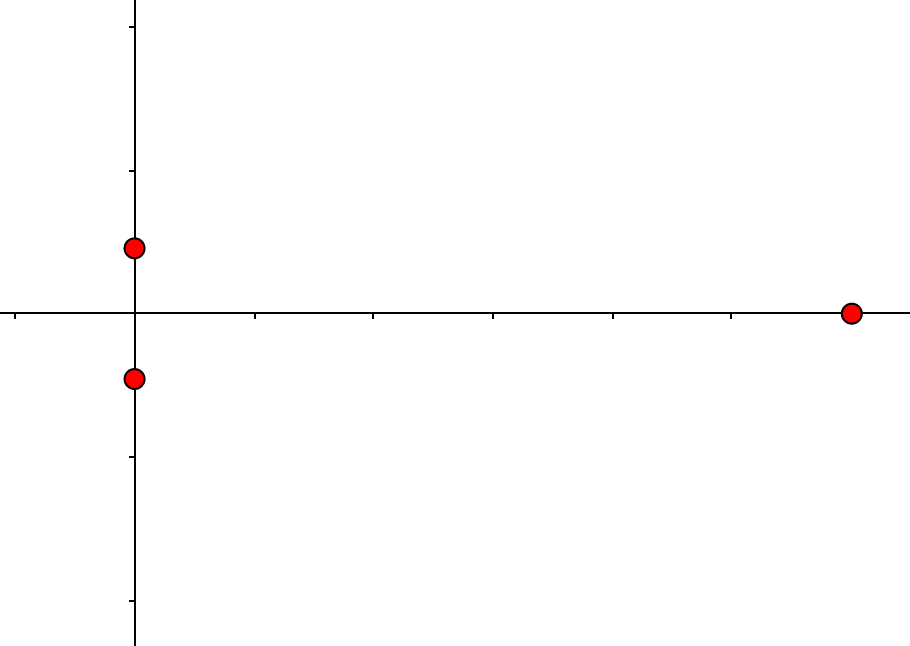}}     
     {\includegraphics[width=0.3\columnwidth]{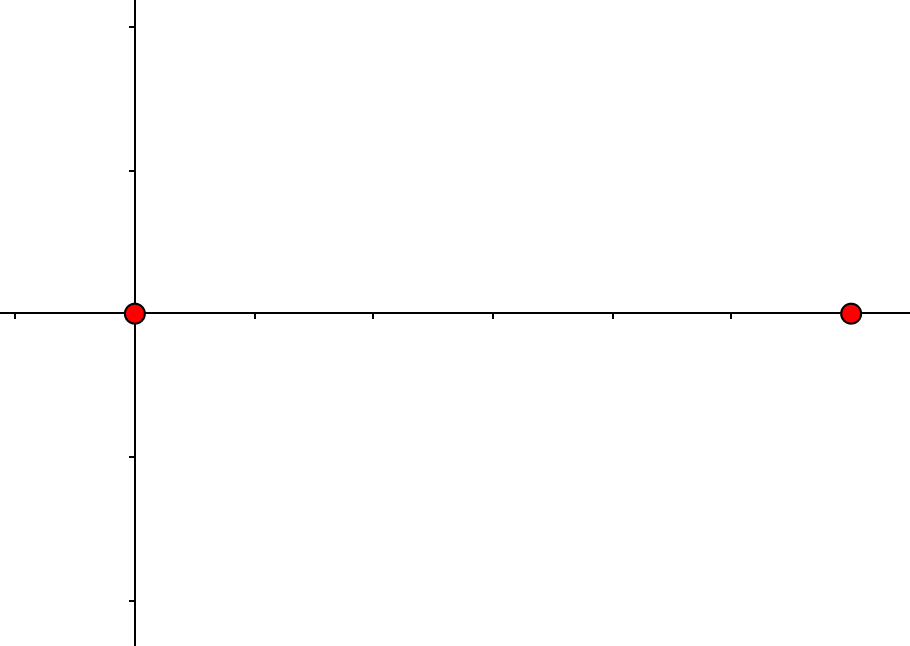}}
   \caption{ For $g(x)=x$, the critical points of $f_t$ are the roots of $4x^3-12x^2-t$. 
   From left to right, the critical points  for $f_1$, $f_{.5}$, and $f_0$ are plotted in the complex planes above.
      }
\end{figure}

A stratification of $X=\cc$ 
such  that $f$ is equisingular on each strata is given by 
 $X_0 = \cc\setminus \{X_1,X_2\}$ and $X_1,X_2$.
In the language of Theorem~\ref{thm_main}, we have
\[
\lim_{t\to 0}\Crit(f_t)  
=
\sum_{i\in \{0,1,2\}}n_i\cdot\Crit(g |_{X_i})
=2 \cdot\Crit(g |_{X_1})  +1 \cdot \Crit(g |_{X_2}).
\] 
Note that the second equality follows as  $g |_{X_0}$ has no critical points.
The $n_i$ are precisely the multiplicity as addressed in Example~\ref{ex:specialization}. 
\end{ex}


\begin{ex}
The next example we consider is $f_t=f-tg:\mathbb{C}^3\to\mathbb{C}$ with $X=\mathbb{C}^3$, $f(x,y,z)=x y^2 - (z-x^2)^2$, and $g$ a general linear function.
The ideal of the variety of critical points of $f$ is generated by the thee partial derivatives of $f$. This ideal has a primary decomposition given by 
$\langle z,y^2,xy, x^2  \rangle$ and 
$\langle y, x^2-z	\rangle$. 
Geometrically, this primary decomposition corresponds to the origin $P$ and a parabola $C$ through the origin.  
An equisingular decomposition of $X$ with respect to $f$ is given by   $X_0 = X\setminus C$, $X_1 = C\setminus \{P\}$ and $X_2 = \{P\}$.  
For general $t$ the function $f_t$ has three critical points, and as $t$ is taken to zero two of the points go to the origin while the third goes to a different point $Q$ in $C$. 
The point $Q$ is the critical point of $g|_C$. 
In the language of Theorem~\ref{thm_main}, we have
\[
\lim_{t\to 0}\Crit ( f_t) = \sum_{i=0,1,2} n_i\cdot \Crit(g |_{X_i}) 
= 1\cdot \Crit(g|_{X_1})+ 2  \cdot \Crit(g|_{X_2})
= 1\cdot Q+ 2 \cdot P.
\]

\end{ex}


\subsection{Semidefinite programming and convex algebraic geometry}\label{ss:SDP}
Semidefinite programming (SDP) is a subfield of convex optimization and has been studied through the lens of algebraic geometry \cite{BPT2013}. 
The aim of an SDP is to  optimize a linear objective function over a 
convex set called a {spectrahedron},
which is the intersection of the cone of positive semidefinite symmetric matrices with an affine space. 

Let $\mathcal{S}^n$ denote the set of $n\times n$  real symmetric  matrices and 
denote the set of $n\times n$ positive semidefinite matrices by $\mathcal{S}^n_+$. 
A set $S\subset \mathbb{R}^m$ is a \emph{spectrahedron} if it has the form
\[
S = \{  ( x_1, \dots , x_m )\in \mathbb{R}^m : A_0+\sum _{i=1}^m A_i x_i \in \mathcal{S}^n_+ \},
\]
for some given symmetric matrices $A_0,A_1,\dots , A_m \in \mathcal{S}^n$.
The \emph{algebraic boundary} of a spectrahedron $S$ is the complex hypersurface given by 
\[
\partial S : =\{ ( x_1, \dots , x_m ) \in \mathbb{C}^n : \det (A_0+\sum _{i=1}^m A_i x_i)=0		\}.
\]
An algebraic approach to SDP is to study the critical points of a linear function on $\partial S$ and to determine the algebraic degree of this optimization problem \cite{MR2496496,MR2546336}.

\begin{ex}[Elliptic curve algebraic boundary]
Consider the spectrahedron 
\[
S = \{( x, y  )\in \mathbb{R}^2 :
\begin{bsmallmatrix}
x+1& 0& y\\
 0& 2  &-x-1\\
  y& -x-1& 2
\end{bsmallmatrix}
 \in \mathcal{S}^3_+ 
 \},
\]
which has an algebraic boundary defined by the elliptic curve 
\[
\partial S= \{
(x,y)\in \mathbb{C}^2 : -x^3-3x^2-2y^2+x+3 = 0
\}.
\]
For an illustration of the real points on the algebraic boundary and a description of the spectrahedron $S$, see \cite[Example 2.7]{BPT2013}.
In the following, we take $X$ to be the algebraic boundary $\partial S$, which is smooth.
Let $g:X\to\mathbb{C}$ denote a general linear function and let $f:X\to\mathbb{C}$ be the projection given by $f(x,y)=-x$.
For $t=0$, the function $f$ has three critical points, which are the three points $X_1,X_2,X_3$ of the curve intersected with the $x$-axis. 
On the other hand, for a general $t$ the general linear function $f_t=f-tg:X\to\mathbb{C}$ has four critical points. 

As we take $t$ to zero, Figure~\ref{fig:SDPellipticCurve} suggests one critical point of $f_t$ goes to~infinity. 
To prove this, by Corollary~\ref{cor:inf}, it suffices to determine 
\eqref{eq:inf} equals one.
This follows as  $\chi(X)=-1$,  $n_i=1$, $\big|\Crit(g|_{X_{i}})\big|=1$,
and
\[
1=-(\chi(X)-3)-\sum_{i\in \{1,2,3\}}n_{i}\cdot \big|\Crit(g|_{X_{i}})\big|=-1(-1-3)-(1+1+1).
\]
In the previous equation we subtract 3 from $\chi(X)$ because a general linear function intersects $X$ at three points.

 \begin{figure}[htb!]
 \centering
   \begin{picture}(153,183)
     {\includegraphics[width=0.5\columnwidth]{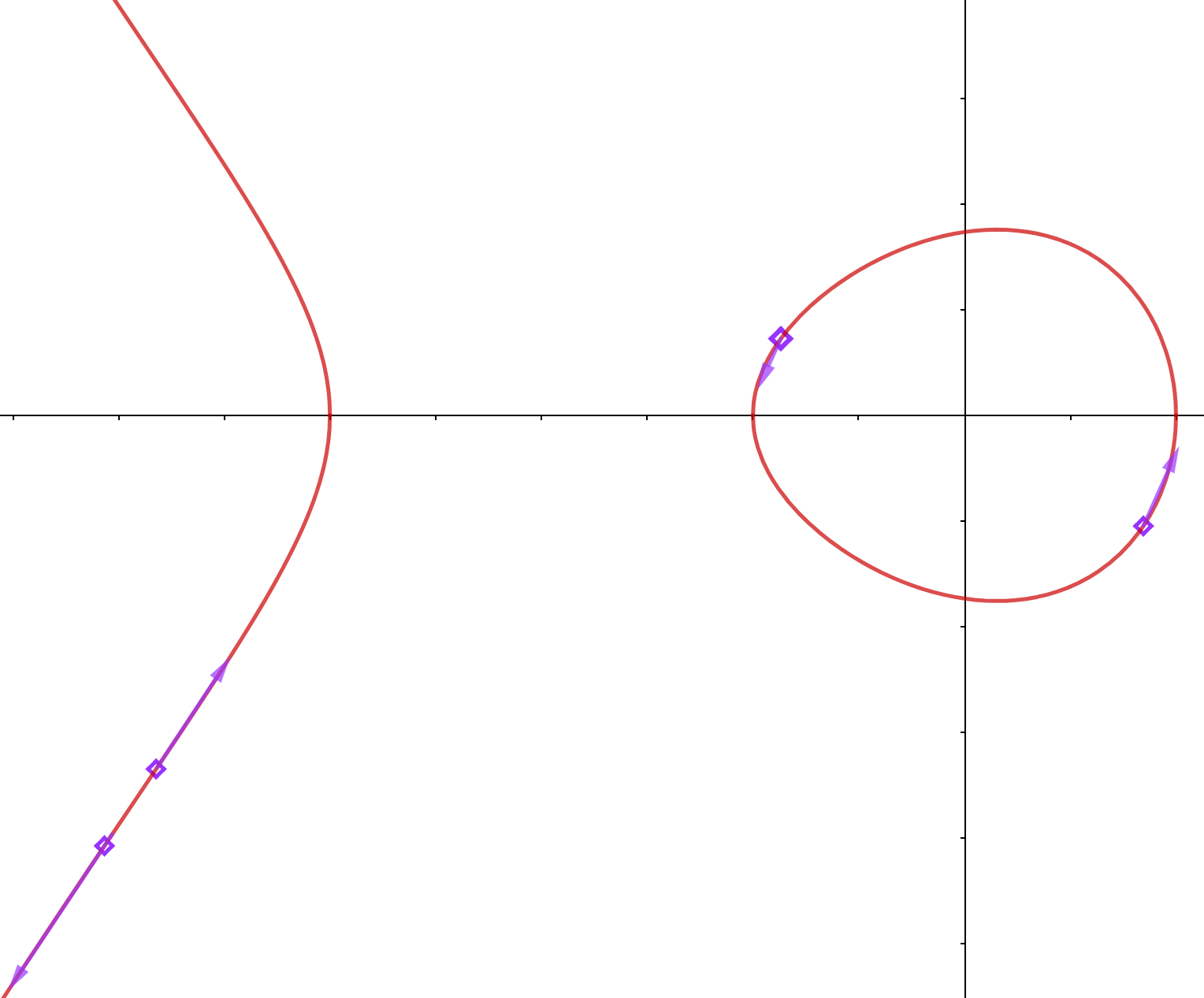}}
   \end{picture}
	\newline
   \caption{
  For $t=1$, the critical points for the general linear function $f(x,y)-tg(x,y)=-x-t(1.5x-0.92y)$ are plotted as purple dots on the elliptic curve $X$. 
  As $t$ is taken to zero, three of the four critical points approach the $x$-axis and one goes to infinity. 
      } \label{fig:SDPellipticCurve}
\end{figure}

\end{ex}

\subsection{Euclidean distance degree}\label{ss:ED}
The Euclidean distance degree (ED degree) \cite{DHOST} of an affine algebraic subvariety $X$ of $\mathbb{C}^n$ 
is defined as the number of critical points 
of the squared Euclidean distance function $d_u(x) := \sum_{i=1}^n(x_i-u_i)^2$
on $X_{\reg}$ for generic $u=(u_1,\dots,u_n)$.
When $X\cap\mathbb{R}^n$ is smooth and compact, the closest point will be a critical point and a solution to the nearest point problem. 
Results on Euclidean distance degrees have a hypothesis requiring genericity of the data point $u$ \cite{AH2018,AH,BD2015, MR3789441,MR3996403, MR3563098,MRW2018} or study discriminant loci \cite{Hor2017}. 
Our results allow us to handle situations when the data is not generic. 
Instances of nongeneric behavior include when the data may be sparse as in Example~\ref{ex:sparse} or satisfy some algebraic property like in Example~\ref{ex:rank}. 

With  generic noise $\epsilon\in\mathbb{C}^n$ and  arbitrary data $u$, 
the data $u+\epsilon$ is generic. 
In the context of distance geometry, 
Theorem~\ref{thm_main} describes what happens to the set of critical points of $d_{u+t\epsilon}$ on $X_{\reg}$ as $t$ is taken to zero.

Let $X$ denote an  subvariety of $\mathbb{C}^n$ with a Whitney stratification $\{S_i\}_{i \in \Lambda}$.
For arbitrary data $u\in\mathbb{C}^n$, generic $\epsilon\in \mathbb{C}^n$, and
$t\in \mathbb{C}$,
consider the squared distance function
\begin{align*}
d_{u+t\epsilon}(x)&=\sum_{i=1}^n  ( x_i - (u_i  + t \epsilon_i ))^2\\
&=\sum_{i=1}^n x_i^2 -2 \sum_{i=1}^n   u_i x_i - 2 \sum_{i=1}^n  t \epsilon_i  x_i +  \sum_{i=1}^n (u_i + t \epsilon_i)^2\\
&=\sum_{i=1}^n (x_i -u_i)^2 - 2 \sum_{i=1}^n  t \epsilon_i  x_i +  \sum_{i=1}^n (u_i + t \epsilon_i)^2-\sum_{i=1}^ nu_i^2\\
&=d_{u}(x)- tg(x) +  c
\end{align*}
with
\begin{equation}\label{eq:epsilonG}
g(x)=2 \sum_{i=1}^n  \epsilon_i  x_i
\end{equation}
 and
 $c=\sum_{i=1}^n (u_i + t \epsilon_i)^2-\sum_{i=1}^ nu_i^2$.
The set of critical points does not depend on $c$ because $c$ is constant with respect to $x$.
So the critical points of $d_{u+t\epsilon}$
 coincide with those of $d_u-tg$. 
Moreover, since  $\epsilon$ is generic, we have
 $g$ is a generic linear function and Theorem~\ref{eq_main} applies to $d_u-tg$.
	
\begin{ex}[Sparse data]\label{ex:sparse}
Consider the curve $X$ in $\mathbb{C}^2$ defined by $x^2+y^2=1$ and the squared distance function 
from the point $p_t=(t\epsilon_1,t \epsilon_2)\in\mathbb{C}^2$, which is 
\[
f_t(x)=( x - t \epsilon_1 )^2+( y - t \epsilon_2 )^2.
\]
When $t$ is generic $f_t$ has two critical points. When $t=0$, $p_0$ is the origin and every point in the curve is a critical point of $f_0$. 
In terms of Theorem~\ref{thm_main}, we have: 
\[
\lim_{t\to 0}\Crit(f_t)= 1\cdot \Crit(g|_X ),
\]
with $\Crit(g|_X )$ consisting of two points.

\begin{ex}[Cardioid Curve]\label{ex:cardioid}
Let $X$ denote the cardioid curve in Figure~\ref{fig:dataZero}, which has a singular point at the origin $P_1$. 
The function $f_t(x,y)=x^2 +y^2 - t (\epsilon_1 x +\epsilon_2 y)$ has three critical points for general $t$ and general $\epsilon=(\epsilon_1,\epsilon_2)$.
Moreover, these critical points coincide with those of the distance function $d_{t\epsilon}$, which are illustrated in Figure~\ref{fig:dataZero} with $\epsilon=(3.12,3.34)$.
The function $f_0:X\to\mathbb{C}$ only has  two  isolated critical points $P_1, P_2$, and Theorem \ref{thm_main} specializes to 
\[
\lim_{t\to 0}\Crit(f_t)= 2 P_1+1 P_2.
\]
\end{ex}

\end{ex}

\begin{ex}[Eckart-Young and low rank data]\label{ex:rank}
In this example, we take $X$ to be the eight dimensional singular hypersurface in $\mathbb{C}^{3\times 3}$ defined by 
$\det [{x_{i,j}}]_{3,3}=0$ consisting of $3\times 3$ matrices of rank at most two.
By the Eckart-Young Theorem,  
the ED degree of $X$ is known to be three.
Moreover, a Whitney stratification of $X$ is given by the rank condition,
i.e., $X$ has a regular stratum consisting of matrices of rank exactly $2$ and the singular locus consists of two strata corresponding to matrices of rank one and zero respectively. 

Consider the following four data matrices

\[
u_1=\left[
\begin{matrix}
3 & 0 & 0\\
0 & 2 & 0\\
0 & 0 & 1
\end{matrix}
\right]
\quad 
u_2=\left[
\begin{matrix}
2 & 0 & 0\\
0 & 1 & 0\\
0 & 0 & 0
\end{matrix}
\right]
\quad 
u_3=\left[
\begin{matrix}
2 & 0 & 0\\
0 & 2 & 0\\
0 & 0 & 1
\end{matrix}
\right]
\quad
u_4=\left[
\begin{matrix}
1 & 0 & 0\\
0 & 0 & 0\\
0 & 0 & 0
\end{matrix}
\right].
\]
Each distance function $d_{u_i+t\epsilon}$ exhibits different limiting behavior among the sets of critical points as $t\to 0$ which we investigate using homotopy continuation methods \cite{MR3563098}. 

For $d_{u_1+t\epsilon}$, the set of three critical  points (corresponding to ED degree of $X$ is three)
converges to the set of three distinct critical points on $X_{\reg}$ given by 
$
\begin{bsmallmatrix}
0& 0& 0\\			 0& 2  &0\\			  0& 0& 1
\end{bsmallmatrix}$, 
 $\begin{bsmallmatrix}
3& 0& 0\\			 0& 0  &0\\			  0& 0& 1
\end{bsmallmatrix}$,~%
$\begin{bsmallmatrix}
3& 0& 0\\			 0& 2  &0\\			  0& 0& 0
\end{bsmallmatrix}$.

The stratified critical locus of $d_{u_2}$ consists of three isolated points, two of which are in the singular locus of $X$. 
Moreover, the set of three critical points of $d_{u_2+t\epsilon}$ converges as $t$ goes to zero to the point 
$\begin{bsmallmatrix}
2& 0& 0\\			 0& 1  &0\\			  0& 0& 0
\end{bsmallmatrix}$
 on the regular locus $X_{\reg}$ 
 and the previously mentioned two points 
  $\begin{bsmallmatrix}
2& 0& 0\\			 0& 0  &0\\			  0& 0& 0
\end{bsmallmatrix}$,
  $\begin{bsmallmatrix}
0& 0& 0\\			 0& 1  &0\\			  0& 0& 0
\end{bsmallmatrix}$
 in the rank one stratum of the singular locus of $X$.

The distance function $d_{u_3}$
has a positive dimensional critical locus given by the union of an isolated regular point and the quadratic curve $Q\subset X_{\reg}$ 
given by
the set of matrices of the form 
$\begin{bsmallmatrix}
a& b& 0\\			 b& 2-a  &0\\			  0& 0& 1
\end{bsmallmatrix}$
with $a(2-a)=b^2$.
The limit set of critical points of $d_{u_3+t\epsilon}$ 
has three distinct points in $X_{\reg}$, one given by 
  $\begin{bsmallmatrix}
2& 0& 0\\			 0& 2  &0\\			  0& 0& 0
\end{bsmallmatrix}$,
and the other two 
being 
contained in $Q$.
These two points correspond to $\Crit(g|_Q)$ where $g$ is a general linear function given by $\epsilon$ as in \eqref{eq:epsilonG}.

For $d_{u_4+t\epsilon}$, the limit of the set of critical points 
consists of one point at the origin with multiplicity one, and another point 
  $\begin{bsmallmatrix}
1& 0& 0\\			 0& 0  &0\\			  0& 0& 0
\end{bsmallmatrix}$
with multiplicity two 
in the rank one stratum.
These respective multiplicities correspond to coefficients $n_i$ 
in Theorem~\ref{thm_main}.

\begin{remark}
In \cite{MRW3}, we studied the number of critical points in the smooth projective case by perturbing the squared Euclidean distance function by a general quadratic function. In the projective setting, no points go to infinity, so we have an equality there. In contrast, in the above examples the emphasis is on perturbing the squared Euclidean distance function with a linear function, and we do not assume the variety to be smooth. 
\end{remark}

\end{ex}

\bibliographystyle{abbrv}
\bibliography{REF_ED_degree_Morse_Pencils}
\end{document}